\documentclass[a4paper,reqno]{amsart}
\usepackage{amsmath, amsfonts, amssymb}
\usepackage{enumerate, tensor, dsfont, mathtools}
\usepackage{enumitem}
\usepackage{graphicx}
\usepackage{tikz}
\usetikzlibrary{arrows, automata}

\usepackage{float}
\usepackage{verbatim}
\usepackage{color}

\numberwithin{equation}{section}

\newtheorem{theo}{Theorem}[section]
\newtheorem{lma}[theo]{Lemma}
\newtheorem{cor}[theo]{Corollary}
\newtheorem{defn}[theo]{Definition}
\newtheorem{cond}[theo]{Condition}
\newtheorem{prop}[theo]{Proposition}

\newtheorem{rem}[theo]{Remark}

\DeclareMathOperator{\R}{\mathbb{R}}

\DeclareMathOperator{\N}{\mathbb{N}}
\DeclareMathOperator{\E}{\mathbb{E}}

\DeclareMathOperator{\Lip}{Lip}

\DeclareMathOperator{\ess}{ess}

\newcommand{\vertiii}[1]{{\left\vert\kern-0.25ex\left\vert\kern-0.25ex\left\vert #1 
    \right\vert\kern-0.25ex\right\vert\kern-0.25ex\right\vert}}
\allowdisplaybreaks
\newcommand\given[1][]{\:#1\vert\:}

\renewcommand{\epsilon}{\varepsilon}

\title[The box-counting dimension of random self-affine sets]{The box-counting dimension of random box-like self-affine sets}

\author{Sascha Troscheit}
\address{Sascha Troscheit\\Department of Pure Mathematics\\University of Waterloo\\200 University Ave W\\Waterloo\\ON\\N2L 3G1\\Canada.}
\date{\today}
\email{stroscheit@uwaterloo.ca}
\urladdr{http://www.math.uwaterloo.ca/~strosche/}
\thanks{The author initially supported by EPSRC Doctoral Training Grant EP/K503162/1, and thanks Kenneth Falconer and Mike Todd for many helpful discussions.
A significant proportion of this manuscript was prepared while visiting the Instituto Nacional de Mathem\'atica Pura e Aplicada (IMPA) and Universidade Federal do Rio de Janeiro (UFRJ) in July/August 2015. The author thanks both IMPA and UFRJ for their hospitality and acknowledges the financial support of IMPA and the EU Marie-Curie IRSES Brazilian-European partnership in Dynamical Systems (FP7-PEOPLE-2012-IRSES 318999 BREUDS)}

\begin{document}
\begin{abstract}
In this paper we study two random analogues of the box-like self-affine attractors introduced by Fraser, which themselves were generalisations of Sierpi\'nski carpets. We determine the almost sure box-counting dimension for the homogeneous random case ($1$-variable random), and give a sufficient condition for the almost sure box-counting dimension to be the expectation of the box-counting dimensions of the corresponding deterministic attractors. Furthermore, we find the almost sure box-counting dimension of the random recursive model ($\infty$-variable), which includes affine fractal percolation. 
\end{abstract}


\maketitle

\section{Introduction}
Self-affine sets, both deterministic and random, have received substantial attention over the past three decades. Far from being an exhausted topic, self-affine sets are still not well understood but much progress has been made in recent years. Of special interest are the dimensional and measure theoretic properties of these sets, analysing their geometric scaling properties. Generally speaking, one would expect the Hausdorff, packing and box-counting dimension to coincide with a quantity called the `affinity dimension'. Recent research has focussed on conditions under which the dimensions coincide and under which circumstances there is a `dimension drop'.
In this paper we establish the almost sure box-counting dimension of two random models of self-affine box-like carpets in the sense of Fraser~\cite{Fraser12}. Before stating our results in detail, which include homogeneous, random recursive, and fractal percolation of self-affine carpets, we give a brief history of both deterministic and random self-affine sets.

Let $(X,d)$ be a (complete) compact metric space. Let $f_{i}:X\to X$ be strictly contracting self maps for some finite index set $i\in\mathcal{I}_{0}$. The (deterministic) attractor $F$ of the iterated function system (IFS) $\mathbb{I}_{0}=\{f_{i}\}_{i\in\mathcal{I}_{0}}$ is the unique compact non-empty set satisfying the invariance (see Hutchinson~\cite{Hutchinson81}),
\begin{equation}
F=\bigcup_{i\in\mathcal{I}_{0}}f_{i}(F).
\end{equation}
Restricting the functions $f_{i}:\R^{d}\to\R^{d}$ to contracting similitudes, one obtains the class of self-similar sets. Much is known about these and we refer the reader to Falconer~\cite{FractalGeo} and the many references contained therein.
Here we are specifically concerned with the class of self-affine carpets. These are the attractors of finite IFSs where $f_{i}$ are affine contractions on $\R^{2}$, i.e.\ they are of the form 
\(
f_{i}\left(\mathbf{x}\right)=\mathbf{M}_{i}\mathbf{x}+\mathbf{v}_{i}
\)
for $\mathbf{x},\mathbf{v}_{i}\in\R^{2}$ and real-valued non-singular $2$ by $2$ matrices $\mathbf{M}_{i}$ with $\lVert\mathbf{M}_{i}\rVert<1$.

For a comprehensive survey of deterministic self-affine sets we refer the reader to Falconer~\cite{SelfAffineSurveyFalconer} but briefly comment that the study of self-affine carpets started with Bedford~\cite{Bedford84} and McMullen~\cite{McMullen84} who determined, independently, the Hausdorff and box-counting dimension of carpets with mappings that map the unit square onto subrectangles obtained by dividing the unit square into an $n\times m$ grid, see Figure~\ref{BedfordFengFraserPic}~(left-most).
Several generalisations were considered by Lalley and Gatzouras~\cite{Gatzouras92}, Bara\'nski~\cite{Baranski07}, and Feng and Wang~\cite{Feng05}, analysing their dimensional properties.
Fraser~\cite{Fraser12},~\cite{Fraser15pre?} extended the class to self-affine carpets with IFSs mapping the unit square into rectangles that are still aligned with the coordinate axes, see Figure~\ref{BedfordFengFraserPic}.

\begin{figure}[t]
\begin{center}
\setlength{\fboxsep}{-1.5pt}
\fbox{\includegraphics[width=0.30\textwidth]{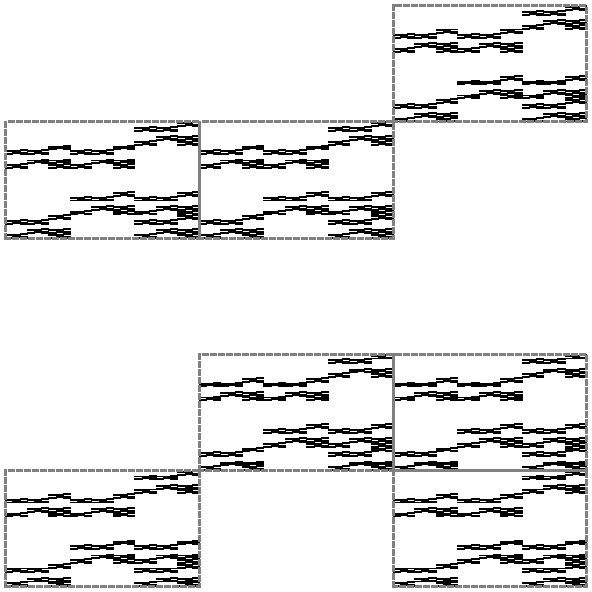}}\hfill
\fbox{\includegraphics[width=0.30\textwidth]{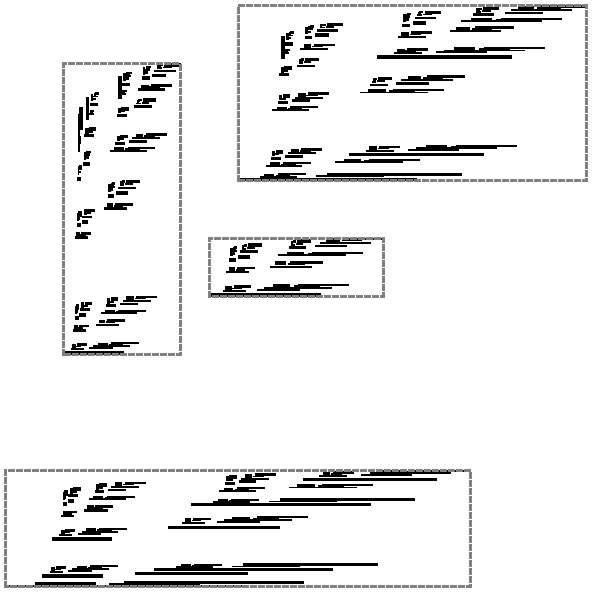}}\hfill
\fbox{\includegraphics[width=0.30\textwidth]{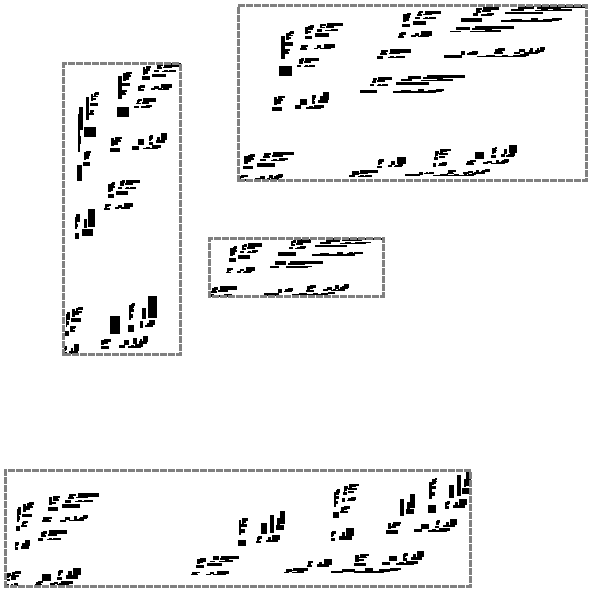}}\hfill\mbox{}
\caption{From left to right: Bedford-McMullen carpet (left), Feng-Wang carpet (middle) and Fraser carpet (right) showing the difference in the introduction of just one map with a reflection.}
\label{BedfordFengFraserPic}
\end{center}
\end{figure}

Several random variants have been considered; Falconer~\cite{Falconer88} considered deterministic self-affine sets with randomly chosen translation vectors and showed that the Hausdorff dimension coincides almost surely (with respect to the chosen translation vectors) with the affinity dimension under some minor conditions on the norm of the defining matrices, see also Solomyak~\cite{Solomyak98} for a more general result.
The affinity dimension can be considered the `best guess' for the Hausdorff, packing and box-counting dimension of self-affine sets and it is of major current interest to establish exactly when these notions do, or do not, coincide.
Jordan, Pollicott and Simon~\cite{Jordan06a}, and Jordan and Jurga~\cite{Jordan14a} studied limit sets with random perturbations of the translation (`noise') at every level of the construction, and found the same almost sure coincidence with the affinity dimension.
Fraser and Shmerkin~\cite{Fraser15pre??} considered a Bedford-McMullen construction with random translation vectors that keep the column structure intact.
Under these conditions they showed that the Hausdorff dimension is strictly less than the affinity dimension, a situation known as a `dimension drop'.

A $1$-variable (homogeneously random) version of Bedford-McMullen carpets was considered by Gui and Li~\cite{Gui08}. 
Here a subdivision $n,m$ of the unit square is fixed and the attractor is created by assigning a probability to all choices of rectangles and then choosing an IFS at every step of the construction independently according to those probabilities. The authors found that in this setting the almost sure Hausdorff and box-counting dimensions equal the mean of the dimensions of the individual deterministic attractors. We will show in Corollary~\ref{ExpectationCor} that this holds in a more general setting for the box-counting dimension, but fails e.g.\ by simply choosing different subdivisions $n_{i}$ and $m_{i}$ for the Bedford-McMullen type IFSs. 
In a later article, Gui and Li~\cite{Gui10} were looking at a similar $1$-variable set up that allowed the subdivisions to vary at different steps in the construction determining the Hausdorff and box-counting dimension and gave sufficient conditions for positive Hausdorff measure. 
Luzia~\cite{Luzia11} considered a $1$-variable construction of self-affine sets of Lalley-Gatzouras type and determined the Hausdorff dimension of these. At this point we refer the reader also to Rams~\cite{Rams10} which gave a more general approach for the determining the Hausdorff dimension in these $1$-variable schemes.
Gatzouras and Lalley~\cite{Gatzouras94} on the other hand were interested in percolation of Bedford-McMullen carpets, which are $\infty$-variable random IFS, also covered in this article.

 J\"arvenp\"a\"a et al.~\cite{Jarvenpaa14a},~\cite{Jarvenpaa15},~\cite{Jarvenpaa15a} also considered a general model (code-tree fractals) that overlaps somewhat with the random model in this paper. However, they consider random translations in their construction and recover almost sure coincidence with the affinity dimension. 
Here we compute the box-counting dimension of random box-like sets without necessarily randomising the translation vectors and thus our results can be used to determine conditions for which there is a `dimension drop' where Hausdorff and affinity dimension do not coincide almost surely.

One of the key elements of our proofs is a recent result obtained by Troscheit~\cite{Troscheit15} stating that self-similar $1$-variable and $\infty$-variable random graph directed (RGDS) systems have equal Hausdorff and box-counting dimension almost surely, see also Roy and Urba\'{n}ski~\cite{Roy11} for a similar approach. 
We use this fact when considering the projections of random self-affine box-like sets onto the horizontal and vertical axes since the projections are self-similar RGDSs.

This article is structured as follows; In Section~\ref{notation} we introduce most basic notation for random iterated function systems common to $1$-variable and $\infty$-variable box-like self-affine carpets. Section~\ref{1varSect} contains our results for random homogeneous ($1$-variable) box-like self-affine carpets and Section~\ref{inftyVarRes} contains our results for random recursive, or $\infty$-variable, carpets. This is followed by some examples in Section~\ref{sec:examples}, while all proofs are contained in Section~\ref{proofs}.

\section{Notation and basic definitions}\label{notation}


We begin by recalling the definition of the box-counting dimension of a set.
\begin{defn}
Let $E\subset X$ be a totally bounded subset of a complete metric space $X$. Denote by $N_{\delta}(E)$ the least number of sets of diameter at most $\delta>0$ needed to cover $E$. The \emph{upper box-counting dimension} and \emph{lower box-counting dimension} are defined, respectively, by
\[
\overline\dim_{B}(E)=\limsup_{\delta\to0}\frac{\log N_{\delta}(E)}{-\log\delta}
\]
and
\[
\underline\dim_{B}(E)=\liminf_{\delta\to0}\frac{\log N_{\delta}(E)}{-\log\delta}.
\]
If $\overline\dim_{B}E=\underline\dim_{B}E$ we will refer to the \emph{box-counting dimension} $\dim_{B}E$.
\end{defn}
Note that the definition does not change if $N_{\delta}(E)$ is substituted by $M_{\delta}(E)$, the number of grid squares in a $\delta$-mesh that intersect $E$. We make use of this in proving upper bounds. More equivalent definitions of the box-counting dimension, and the definition of the Hausdorff dimension, can be found in Falconer~\cite{FractalGeo}.

We further comment that for topological reasons the packing dimension equals the upper box-counting dimension for all the constructions we will consider. Therefore our results not only determine the box-counting but also the packing dimension. This follows directly from Corollary~3.10 in~\cite{FractalGeo}.

\subsection{Random Iterated Function Systems}
Let ${\mathbb{I}}=\{\mathbb{I}_{i}\}_{i\in\Lambda}$ be a finite collection of IFSs indexed by $i\in\Lambda=\{1,\dots,N\}$, where $\mathbb{I}_{i}=\{f_{i}^{1},\dots,f_{i}^{\#\mathbb{I}_{i}}\}$ and the affine maps $f_{i}^{j}:\R^{2}\to \R^{2}$ are strict contractions. 
With the collection ${\mathbb{I}}$ we associate a (non-degenerate) probability vector $\mathbf{p}=\{p_{i}\}$, i.e.\ $p_{i}>0$ for all $i\in\Lambda$ and $\sum p_{i}=1$.
We call the pair $(\mathbb{I},\mathbf{p})$ a random iterated function system (RIFS).
In Sections~\ref{1varSect} and \ref{inftyVarRes} we will define the random sets associated with them.

The self-affine sets we are considering were introduced by Fraser~\cite{Fraser12} and are known as box-like self-affine carpets. 
\begin{defn}\label{def:frasercarpet}
For a given $i$, let $f_{i}^{j}:\R^{2}\to\R^{2}$ be of the form
\[
f_{i}^{j}(\mathbf{x})=\begin{pmatrix}
a_{i}^{j} 	& 	0		\\
0		&	b_{i}^{j}
 \end{pmatrix}
 \mathbf{Q_{i}^{j}}
 \begin{pmatrix}
 x_{1}\\
 x_{2}
 \end{pmatrix}
 +
 \begin{pmatrix}
 u_{i}^{j}\\
 v_{i}^{j}
 \end{pmatrix},
\]
where $0<a_{i}^{j}<1$ and $0<b_{i}^{j}<1$, $\mathbf{x}=(x_{1},x_{2})$,  and $u_{i}^{j},  v_{i}^{j}\in\R$ are such that the unit square $\Delta=[0,1]^{2}$ is mapped into itself, that is $f_{i}^{j}(\Delta)\subset \Delta$ and
\[
\mathbf{Q}_{i}^{j}\in\left\{\begin{pmatrix}\pm 1&0\\0&\pm1\end{pmatrix},\begin{pmatrix}0&\pm1\\\pm1&0\end{pmatrix}\right\}.
\]

If all maps $f_{i}^{j}\in\mathbb{I}_{i}$ satisfy the criteria above we call the IFS $\{f_{i}^{j}\}_{j=1}^{\#\mathbb{I}_{i}}$ \emph{box-like}.
If all IFSs $\mathbb{I}_{i}\in\mathbb{I}$ are box-like we call the RIFS $(\mathbb{I},\mathbf{p})$ \emph{box-like}.
\end{defn}

We remark that the matrices $\mathbf{Q}_{i}^{j}$ represent elements of the symmetry group of isometries $D_{8}$ such that $f_{i}^{j}$ maps the square onto rectangles that are still aligned with the $x$ and $y$ axis.

\begin{defn}
Let $({\mathbb{I}},\mathbf{p})$ be box-like as above. 
If there exist at least two pairs $(i,j)$ and $(k,l)$ with
\[
\mathbf{Q}_{i}^{j}\in\left\{
\begin{pmatrix}
0&\pm1\\\pm1&0
\end{pmatrix}\right\}, \text{ and } 
\mathbf{Q}_{k}^{l}\in\left\{
\begin{pmatrix}
\pm1&0\\0&\pm1
\end{pmatrix}\right\},
\]
we call $({\mathbb{I}},\mathbf{p})$ \emph{non-separated}.
Otherwise, we call $({\mathbb{I}},\mathbf{p})$ \emph{separated}. 
\end{defn}
We note that this differs from~\cite{Fraser12} by only requiring one of the IFSs in $\mathbb{I}_i\in\mathbb{I}$ to have one map with a differing diagonal structure. This is the appropriate analogue to consider in the random setting as every IFS is chosen infinitely often with full probability, providing the necessary `mixing' of projections.

Results in dimension theory usually require some assumptions on the level of overlap. We introduce the random analogue of the condition introduced by Feng and Wang~\cite{Feng05}.

\begin{defn}\label{defn:UORC}
Let $(\mathbb{I},\mathbf{p})$ be a box-like RIFS. We say that $({\mathbb{I}},\mathbf{p})$ satisfies the \emph{uniform open rectangle condition (UORC)}, if we have, for every $i\in\Lambda$, 
\[
f_{i}^{j}\left(\mathring\Delta\right)\cap f_{i}^{k}\left(\mathring{\Delta}\right)\neq\varnothing\;\;\;\Rightarrow\;\;\; k=j.
\]
Here $\mathring\Delta=(0,1)^{2}$ is the open unit square.
\end{defn}

\subsection{Codings}
To each map $f_{i}^{j}$ we associate a unique symbol $e_{i}^{j}$ to enable us to code points in the random attractor.
We adopt the notation of \emph{arrangement of words} in \cite{Troscheit15} to succinctly write sets of codings.
 
\begin{defn}
Let $\mathcal{G}^{E}=\{e_{i}^{j} \mid  i\in\Lambda, 1\leq j\leq \#\mathbb{I}_{i} \}$, and define the \emph{prime arrangements} $\mathcal{G}$ to be the set of symbols $\mathcal{G}=\{\varnothing,\epsilon_{0}\}\cup\mathcal{G}^{E}$. Clearly both $\mathcal{G}$ and $\mathcal{G}^{E}$ are finite and non-empty.

Define $\beth^{\odot}$ to be the free monoid with operation $\odot$, generators $\mathcal{G}^{E}$, and identity $\epsilon_{0}$, and define $\beth^{\sqcup}$ to be the free commutative monoid with operation $\sqcup$, generators $\beth^{\odot}$, and identity $\varnothing$. We define $\odot$ to be left and right multiplicative over $\sqcup$, and $\varnothing$ to annihilate with respect to $\odot$. That is, given an element $e$ of $\beth^{\odot}$, we get $e\odot\varnothing=\varnothing\odot e=\varnothing$. 
We write $\beth^{*}$ for the set of finite algebraic expressions with elements $\mathcal{G}$ and operations $\sqcup$ and $\odot$. Using distributivity $\beth=(\beth^{*},\sqcup,\odot)$ is the non-commutative free semi-ring with `addition' $\sqcup$ and `multiplication' $\odot$ and generator $\mathcal{G}^{E}$ and we will call $\beth$ the \emph{semiring of arrangements of words} and refer to elements of $\beth^{*}$ as \emph{(finite) arrangements of words}. 
\end{defn}
We use the convention to `multiply out' arrangements of words and write them as elements of $\beth^{\odot}$ separated by $\sqcup$. We omit brackets, where appropriate, replace $\odot$ by concatenation to simplify notation, and for each such reduced arrangement of words $\phi=\varphi_{1}\sqcup\varphi_{2}\sqcup\dots\sqcup\varphi_{k}$ write $\varphi_{i}\in\phi$ to refer to the subarrangements $\varphi_{i}$ that are elements of $\varphi_{i}\in\beth^{\odot}$.

Arrangements of words can be used to describe compact subsets by defining a coding map, mapping elements $\phi\in\beth^{*}$ onto compact subsets $\mathcal{K}(\R^{2})$ recursively.
\begin{defn}
Given an arrangement of words $\phi$ and a compact set $K\in\mathcal{K}(\R^{2})$, we define $f(\phi,K)$ recursively to be the compact set satisfying:
\[
f(\phi,K)=\begin{cases}
f(\phi_{1},K)\cup f(\phi_{2},K),&\text{ if }\phi=\phi_{1}\sqcup\phi_{2};\\
f(\phi_{1}, f(\phi_{2},K)),&\text{ if }\phi=\phi_{1}\odot\phi_{2};\\
f_{i}^{j}(K),&\text{ if }\phi=e_{i}^{j};\\
K,&\text{ if }\phi=\epsilon_{0};\\
\varnothing,&\text{ if }\phi=\varnothing.
\end{cases}
\]
\end{defn}

To each IFS we associate an arrangement of words.

\begin{defn}
Let $W_{i}$ be the arrangement of words that are the letters coding the maps of the IFS $\mathbb{I}_{i}$,
\[
W_{i}=e_{i}^{1}\sqcup e_{i}^{2}\sqcup\dots\sqcup e_{i}^{\#\mathbb{I}_{i}}.
\]
\end{defn}

This representation now allows us to define sets involving the IFSs recursively by right multiplication of $W_{i}$ to existing codings.

\subsection{Projections}
Our results depend on the box-counting dimensions of the orthogonal projections onto the $x$ and $y$ axes. We write $\pi_{x}$ and $\pi_{y}$ to denote these projections, respectively:
\[
\pi_{x}:\R^{2}\to\R\,,\;\;\text{ where }\;\; \pi_{x}\left((z_{1},z_{2})^{\top}\right) = z_{1}\,,\;\;\text{ and}
\]
\[
\pi_{y}:\R^{2}\to\R\,,\;\;\text{ where }\;\; \pi_{y}\left((z_{1},z_{2})^{\top}\right) = z_{2}.
\]
For $e\in\beth^{\odot}$, let $\alpha_{M}(e)=\max_{z\in\{x,y\}}\lvert \pi_{z}(f(e,\Delta))\rvert$ and $\alpha_{m}(e)=\min_{z\in\{x,y\}}\lvert \pi_{z}(f(e,\Delta))\rvert$ be the length of the longest and shortest edge, respectively, of the rectangle $f(e,\Delta)$. We define $\overline{s}(e,F)$ to be the upper box-counting dimension of the projection of a compact set $F\subset\R^2$ onto the line parallel to the longest side of $f(e,\Delta)$, that is
\[
\overline{s}(e,F)=
\begin{cases}
\overline{\dim}_{B}(\pi_{x}F)	,	& \text{if } \lvert \pi_{x}(f(e,\Delta))\rvert \geq \lvert \pi_{y}(f(e,\Delta))\rvert,\\
\overline{\dim}_{B}(\pi_{y}F)	,	& \text{otherwise.}
\end{cases}
\]
Analogously, let $\underline{s}(e,F)$ be the lower box-counting dimension of the projection of $F$. If the box-counting dimension exists we write $s(e,F)$ for the common value. 
Let $\overline{s}^{x}(F)=\overline{\dim}_{B}(\pi_{x}F)$ and $\overline{s}^{y}(F)=\overline{\dim}_{B}(\pi_{y}F)$, with $\underline{s}^{x}(F)$, $\underline{s}^{y}(F)$, $s^{x}(F)$, and $s^{y}(F)$ defined analogously.
We will write $\pi_{e}$ to denote the projection, $\pi_{x}$ or $\pi_{y}$, parallel to the long side of the rectangle $f(e,\Delta)$, choosing arbitrarily if they are equal.

Finally, we define stochastically additive and sub-additive random variables.
\begin{defn}\label{defn:stochadd}
Let $\{X_{i}(\omega)\}_{i}$ be a family of measurable random variables defined on some probability space $\Omega$, with $\omega\in\Omega$ and $i\in\N$. Let $\sigma:\Omega\to\Omega$.
If $X_{i+j}(\omega)\leq X_{i}(\omega) + X_{j}(\sigma^{i} \omega)$ for all $i,j\in\N$ and $\omega\in\Omega$, then $\{X_{i}(\omega)\}_{i}$ is said to be \emph{stochastically sub-additive}. In the case of equality, $\{X_{i}(\omega)\}_{i}$ is said to be \emph{stochastically additive}. The value $\left(X_{i+j}(\omega)-X_{i}(\omega) - X_{j}(\sigma^{i} \omega)\right)^{+}$, where $(x)^{+}=\max\{0,x\}$ is called the \emph{sub-additive defect}.
\end{defn}

\section{Results for {1}-variable self-affine carpets}\label{1varSect}\index{$1$-variable attractor}\index{self-affine set}

Let $\Lambda^{\N}=\{1,\dots,N\}^{\N}$ be the set of all (infinite) sequences with entries in $\Lambda$; we refer to these sequences as (infinite) words or codings and for $w\in\Lambda^{\N}$ we write $w_{i}$ for the $i$-th letter in the sequence, i.e.\ the $i$-th letter in the word $w$. We define the distance between $x,y\in\Lambda^{\N}$ to be $d(x,y)=2^{-\lvert x\wedge y\rvert}$, where $\lvert x\wedge y\rvert=\max\{k \mid x_{i}=y_{i} \text{ for } i\leq k\}$ and $d(x,y)=0$ if $x=y$. 
Then $d$ is a metric and we consider the metric \emph{coding space}, $(\Lambda^{\N},d)$.
We sometimes need finite words and set $\Lambda^{k}$ to be the set of words of length $k$. The set $\Lambda^{0}$ contains only the empty word $\epsilon_{0}$ and the set of all finite codings is $\Lambda^{*}=\bigcup_{k=0}^{\infty} \Lambda^{k}$. For $w\in\Lambda^{*}$ we write $\lvert w\rvert$ to denote the length of the word $w$ and we let $[w]=\{x\in\Lambda^{\N} \mid x_{i}=w_{i}\text{ for }1\leq i\leq\lvert w\rvert\}$ be the \emph{cylinder} given by the coding $w$. With the given metric, the cylinders are clopen sets and are a basis for the topology on $(\Lambda^{\N},d)$.
By the Carath\'eodory extension theorem we can define a measure on these cylinders which extends to a unique Borel measure on $\Lambda^{\N}$. Let $\mu$ be the Bernoulli probability measure induced by the probability vector $\mathbf{p}$, i.e.\ for a cylinder $[w]$, $w\in\Lambda^{*}$ we set
\[
\mu([w])=\prod_{i=1}^{N}p_{i}^{\#\{j \mid w_{j}=i\}}.
\] 
This defines a Borel probability measure in the usual way.
The probability space we are considering in this section is $(\Omega,\mathcal{B}(\Omega),\mu)$, where $\Omega=\Lambda^{\N}$ and $\mathcal{B}(\Omega)$ is the Borel $\sigma$-algebra of $\Omega$. Almost sure results for the $1$-variable setting are with respect to the probability measure $\mu$.

We now define the random set we are investigating in this section. In fact we associate a set $F_{\omega}$ to every $\omega\in\Lambda^{\N}$. Choosing $\omega$ randomly according to $\mu$ gives us the random attractor $F_{\omega}$.

\begin{defn}\label{defn:1varcoding}
The \emph{$k$-level coding} with respect to realisation $\omega\in\Lambda^{\N}$ is 
\[
\mathbf{C}^{k}_{\omega}=W_{\omega_{1}}\odot W_{\omega_{2}} \odot\dots\odot W_{\omega_{k}}\;\;(k\in\N)\;\;\text{ and }\;\; \mathbf{C}^{0}_{\omega}=\epsilon_{0}.
\]
The arrangement of all finite codings $\mathbf{C}_{\omega}^{*}$ is defined  by
\[
\mathbf{C}_{\omega}^{*}=\bigsqcup_{i=0}^{\infty} \mathbf{C}_{\omega}^{i}.
\] 
\end{defn}

Recall that $\sqcup$ represents addition in the semiring $\beth$.

 \begin{defn}
The \emph{$k$-level prefractal $F_{\omega}^{k}$} and the \emph{$1$-variable random box-like self-affine carpet} $F_{\omega}$ are
\[
F_{\omega}^{k}=f(\mathbf{C}^{k}_{\omega},\Delta)=\bigcup_{e\in \mathbf{C}^{k}_{\omega}}f(e,\Delta)\subset\R^{2}
\]
and
\[
F_{\omega}=\bigcap_{k=1}^{\infty}f(\mathbf{C}^{k}_{\omega},\Delta)=\bigcap_{k=1}^{\infty}\bigcup_{e\in \mathbf{C}^{k}_{\omega}}f(e,\Delta)\subset\R^{2},
\]
where $\Delta=[0,1]^{2}$.
\end{defn}
For reasons of non-triviality we assume that each IFS in ${\mathbb{I}}$ has at least one map, with at least one IFS containing two maps. This guarantees that $F_{\omega}$ is almost surely not a singleton.

We define a singular value function for each realisation $\omega\in\Lambda^{\N}$.
\begin{defn}
Let $e\in\beth^{\odot}$, we define the \emph{upper (random) modified singular value function} by
\[
\overline{\psi}_\omega^s(e)=\alpha_{M}(e)^{\overline{s}(e,F_{\omega})}\alpha_{m}(e)^{s-\overline{s}(e,F_{\omega})}.
\]
Let $\overline{\Psi}_\omega^k(s)$ be the sum of the modified singular values over all $k$-level words,
\[
\overline{\Psi}_\omega^k (s)=\sum_{e\in\mathbf{C}^{k}_{\omega}}\overline{\psi}_\omega^s(e).
\]
We let $\underline{\psi}_\omega^s(e)$ and $\underline{\Psi}_\omega^k (s)$ be the  lower  modified singular value function and its sum, defined analogously.
\end{defn}
We will now introduce the last component, the pressure, which relates to the topological pressure of the associated dynamical system.
\begin{defn}
Let $s\in\R^{+}_{0}$, the \emph{upper $s$-pressure} for realisation $\omega\in\Lambda^{\N}$ is given by 
\[
\overline{P}_{\omega}(s)=\overline\lim_{k\to\infty}\left(\overline{\Psi}_\omega^k(s)\right)^{1/k}.
\]
The lower pressure $\underline{P}_{\omega}$ is defined analogously.
\end{defn}

\begin{lma}\label{lma:spressure}
There exists a function $P(s):\R^{+}_{0}\to\R^{+}_{0}$, the \emph{$s$-pressure}, such that $\overline{P}_{\omega}(s)=\underline{P}_{\omega}(s)=P(s)$ for $\mu$-almost every $\omega\in\Omega$. Further, $P(s)$ is continuous and strictly decreasing and there exists a unique $s_{B}\in\R_{0}^{+}$ satisfying,
\begin{equation}\label{unityPressure}
P(s_{B})=1.
\end{equation}
\end{lma}
We note that we are taking the liberty of calling $P$ a \emph{pressure} even though it is more appropriately the \emph{exponential of pressure}. Similarly, if $\overline\psi_\omega^s(e\odot g)=\overline\psi_\omega^s(e)\overline\psi_{\sigma^k\omega}^s(g)$ we may refer to the random variable $\overline\psi_\omega^s(e)$ as additive, where we should formally refer to the logarithm of the random variable as additive.
Section~\ref{proofs} contains the proof of above lemma and our main result for the box-counting dimension of $1$-variable random box-like self-affine carpets.

\begin{figure}[t]
\begin{center}
\setlength{\fboxsep}{-1.5pt}
\fbox{\includegraphics[width=0.30\textwidth]{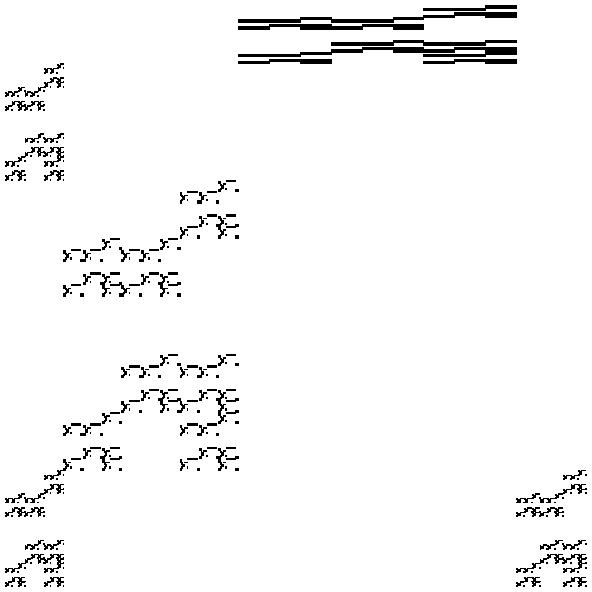}}\hfill
\fbox{\includegraphics[width=0.30\textwidth]{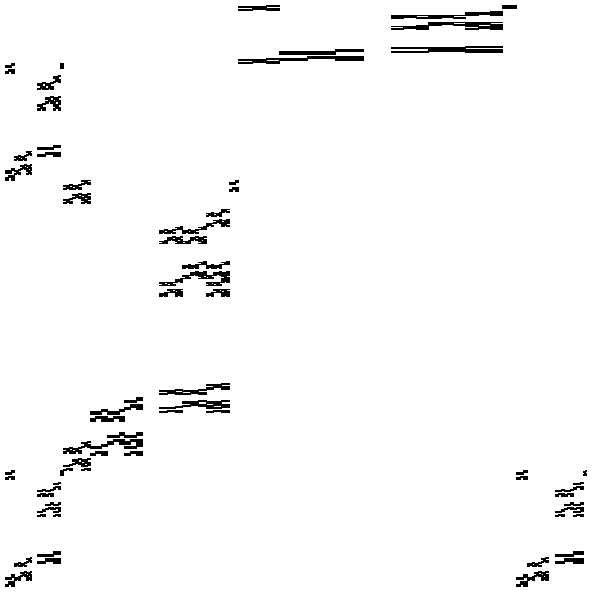}}\hfill
\fbox{\includegraphics[width=0.30\textwidth]{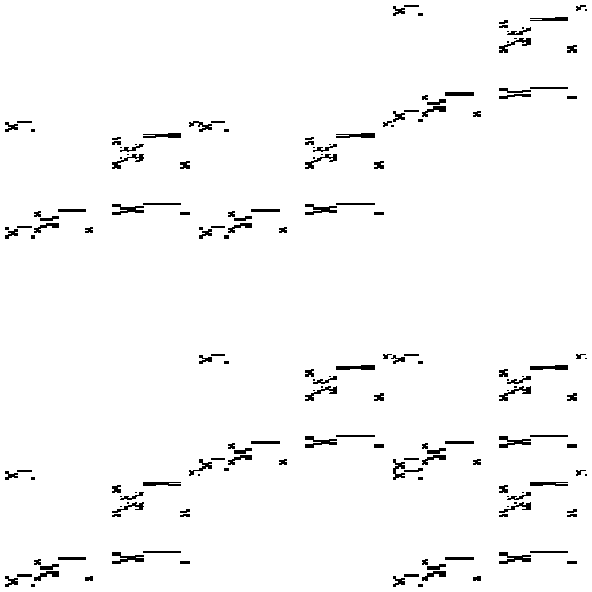}}
\caption{Three random realisations using the maps in Figure~\ref{BedfordFengFraserPic}.}
\label{RandomPic1}
\end{center}
\end{figure}

\begin{theo}\label{thm:mainTheo}
Let $(\mathbb{I},\mathbf{p})$ be a box-like self-affine RIFS that satisfies the UORC. Let $F_{\omega}$ be the associated $1$-variable random box-like self-affine carpet. Then
\begin{equation}\label{mainEq}
\dim_B F_\omega=s_{B},
\end{equation}
for almost every $\omega\in\Lambda^{\N}$, where $s_{B}$ is the unique solution to $P (s_{B})=1$.
\end{theo}

Applying Lemma~\ref{lma:stochAdditivity}, we get the following interesting corollary.
\begin{cor}\label{cor:reducedEqStochAd}
Let $(\mathbb{I},\mathbf{p})$ be a box-like self-affine RIFS that satisfies the UORC and is of the separated type, with $\alpha_{M}(e)=\lvert\pi_{x}f(e,\Delta)\rvert$ for all $e\in\mathbb{I}_{i}\in\mathbb{I}$. Let $F_{\omega}$ be the associated $1$-variable random box-like carpet.
Then $\overline{\psi}_{\omega}^{t}$ is (stochastically) additive and the box-counting dimension of $F_{\omega}$ is almost surely given by the unique $s_{B}$ such that,
\begin{equation}\label{shortenedEq}
\exp\E\left(\log \sum_{e\in W_{\omega_{1}}}\overline\psi_{\omega}^{s_{B}}(e)\right)=1.
\end{equation}
\end{cor}
\begin{proof}
	First note that additivity of $\overline\psi_\omega^s(e)$ implies
	\[
	\log\overline\Psi_\omega^k(s)=\log\sum_{e\in\mathbf{C}^{k}_{\omega}}\overline\psi_{\omega}^{s_{B}}(e)
	=\log\left[\Bigg(\sum_{e\in W_{\omega_{1}}}\overline\psi_{\omega}^{s_{B}}(e)\Bigg)\dots \Bigg(\sum_{e\in W_{\omega_{k}}}\overline\psi_{\sigma^{k-1}\omega}^{s_{B}}(e)\Bigg)\right].
	\]
	Therefore 
	\[
	\lim_{k\to\infty}\frac{1}{k}\log\overline\Psi_\omega^k(s)
	=\lim_{k\to\infty}\frac{1}{k}\sum_{i=1}^k\log\left(\sum_{e\in W_{\omega_{i}}}\overline\psi_{\sigma^{i-1}\omega}^{s_{B}}(e)\right)
	=\E\left(\log \sum_{e\in W_{\omega_{1}}}\overline\psi_{\omega}^{s_{B}}(e)\right),
	\]
	almost surely by Birkhoff's Ergodic Theorem. For $s_B$ satisfying (\ref{shortenedEq}) this limit equals $0$ almost surely and hence $P(s_B))=1$. We can now apply Theorem~\ref{thm:mainTheo} and obtain the desired conclusion.
\end{proof}
Introducing further conditions, we can express the box-counting dimension in terms of the individual attractors. The following corollary to Corollary~\ref{cor:reducedEqStochAd} extends the box-counting dimension result from Gui and Li~\cite{Gui08} which states that for $1$-variable Bedford-McMullen carpets with subdivisions $n,m$ the almost sure box-counting dimension is the mean of the box-counting dimensions of the corresponding deterministic attractors. 
\begin{cor}\label{ExpectationCor}
Let $(\mathbb{I},\mathbf{p})$ be a box-like self-affine RIFS that satisfies the UORC and is of the separated type, with $\alpha_{M}(e)=\lvert\pi_{x}f(e,\Delta)\rvert$ for all $e\in\mathbb{I}_{i}\in\mathbb{I}$. Let $F_{\omega}$ be the associated $1$-variable random box-like carpet. Assume further that
\begin{enumerate}
\item there exists $\eta\in(0,1)$ s.t.\ $\alpha_{m}(e)=\eta$ for all $e\in\mathbb{I}_{i}$ and $i\in\Lambda$,
\item $s^{x}(F_{\omega})=\sum_{i\in\Lambda}p_{i}s^{x}(F_{i,i,\dots})$ almost surely,
\item and the following equality holds:
\[
\E\left(\log \sum_{e\in W_{\omega_{1}}}\alpha_{M}(e)^{s^{x}(F_{\omega})}\right)=\E\left(\log \sum_{e\in W_{\omega_{1}}}\alpha_{M}(e)^{s^{x}(F_{\omega_{1},\omega_{1},\dots})}\right).
\]
\end{enumerate}
Then, almost surely,
\begin{equation}\label{expectationOfIndivid}
\dim_{B}F_{\omega}=\sum_{i\in\Lambda}p_{i}\dim_{B}F_{i,i,\dots}=\E(\dim_{B}F_{\omega_1,\omega_1,\dots}).
\end{equation}
\end{cor}
\begin{proof}
First note that $s^{x}(F_{\omega})$ is constant almost surely. We denote this value by  $s^{x}$ and note from (\ref{shortenedEq}),
\begin{align*}
1&=\exp\E\left(\log \sum_{e\in W_{\omega_{1}}}\alpha_{M}(e)^{s^{x}}\alpha_{m}(e)^{s_{B}-s^{x}}\right)\\
&=\exp\E\left(\log \left(\eta^{s_{B}-s^{x}}\sum_{e\in W_{\omega_{1}}}\alpha_{M}(e)^{s^{x}}\right)\right).
\end{align*}
So
\[
 \eta^{-s_{B}}=\eta^{-s^{x}}\left( \sum_{e\in W_{1}}\alpha_{M}(e)^{s^{x}}\right)^{p_{1}}\cdot\hdots\cdot\left( \sum_{e\in W_{N}}\alpha_{M}(e)^{s^{x}}\right)^{p_{N}}
 \]
but $\eta^{-s^{x}}=\eta^{-\sum_{i\in\Lambda}p_{i}s^{x}(F_{i,i,\dots})}$ almost surely and hence, almost surely,
\[
 \eta^{-s_{B}}=\left( \sum_{e\in W_{1}}(\alpha_{M}(e)/\eta)^{s^{x}(F_{1,1,\dots})}\right)^{p_{1}}\cdot\hdots\cdot\left( \sum_{e\in W_{N}}(\alpha_{M}(e)/\eta)^{s^{x}(F_{N,N,\dots})}\right)^{p_{N}}
 \]
 and
\[
s_{B}=\frac{\sum_{i\in\Lambda}p_{i}\log\left( \sum_{e\in W_{i}}(\alpha_{M}(e)/\eta)^{s^{x}(F_{i,i,\dots})}\right)}{-\log\eta}.
\]
Thus $s_{B}$ is the weighted average of $\dim_{B}F_{i,i,\dots}$.
\end{proof}
On first glance these conditions seem very restrictive. However, note that $1$-variable Bedford-McMullen carpets sharing the same $n,m$  grid subdivision satisfy these conditions (the Gui-Li case). Briefly, this is because 
\[
s^{x}(F_{i,i\dots})=\frac{-\log(\text{\em number of non-empty columns})}{\log(\text{\em column width})},
\]
and 
\[
s^{x}(F_{\omega})=\frac{-\log(\text{\em geometric mean of the number of non-empty columns})}{\log(\text{\em column width})}\;\;\text{(a.s.).}
\]
Further, these conditions are satisfied for much more general separated box-like self-affine RIFS (such as the Lalley-Gatzouras type) if all the individual attractors' projection onto the horizontal have the same box-counting dimension and they contract equally in the direction parallel to the vertical.

Note however, that letting $\alpha_{m}(e)=\eta_{i}$ for every $e\in W_{i}$ is no longer sufficient for the dimension to be the mean of the individual dimensions as Example~\ref{simpleExtensionExample} in Section~\ref{sec:examples} shows.
Another interesting consequence of Theorem~\ref{thm:mainTheo} is obtained for RIFSs such that the modified singular value function is not stochastically additive.
\begin{rem}
Let $(\mathbb{I},\mathbf{p})$ be a box-like self-affine RIFS that satisfies the UORC such that $\overline{\psi}_{\omega}^{s}(e)$ is not stochastically additive. Let $F_{\omega}$ be the associated $1$-variable random box-like carpet.
 Then the almost sure box-counting dimension of the attractor can drop below the least box-counting dimension of the individual attractors, that is there exists $(\mathbb{I},\mathbf{p})$ such that, almost surely,
\[
\dim_{B}F_{\omega}<\min_{i\in\Lambda}\dim_{B}F_{i,i,i,\dots}\,.
\]
See Example~\ref{weirdlyConvergingExample} in Section~\ref{sec:examples}.
\end{rem}

We end this section by commenting that if $s^{x}=s^{y}=1$ a.s.\ the modified singular value function coincides with the singular value function and $\dim_{B}F_{\omega}$ coincides with the natural affinity dimension. 
For the separated case with greatest contraction in the vertical direction it is sufficient to have $s^{x}=1$. 
Conversely, if $s^{x},s^{y}<1$ the almost sure box-counting dimension (and therefore the almost sure Hausdorff dimension) of $F_{\omega}$ will be strictly less than the associated affinity dimension.

\section{Results for $\infty$-variable box-like carpets}\label{inftyVarRes}

In this section we define an infinite code tree and define the $\infty$-variable attractor of a random iterated function system $(\mathbb{I},\mathbf{p})$.
We set $k_{s}=\max_{i\in\Lambda}\#\mathbb{I}_{i}$ and consider the rooted $k_{s}$-ary tree. Each node in this tree we label with a single $i\in\Lambda$, chosen independently, according to the probability vector $\mathbf{p}$. 
We denote the space of all possible labellings of the tree by $\mathcal{T}$ and refer to individual realisations picked according to the induced probability measure, described below, by $\tau\in\mathcal{T}$.
In this full tree we address vertices by which branch was taken; if $v$ is a node at level $k$ we write $v=(v_{1},v_{2},\dots,v_{k})$, with $v_{i}\in\{1,\dots,k_{s}\}$ and root node $v=(.)$.
The levels of the tree are then:
\[
\{(.)\}, \{(1),(2),\dots,(k_{s})\}, \{ (1,1),(1,2),\dots,(1,k_{s}),(2,1),\dots,(k_{s},k_{s})\},\dots\;.
\]
We write $\tau(v)\in\Lambda$ to denote the random letter for node $v$ and realisation $\tau$. Given a node $v$ we define $\sigma^{v}\tau$ to be the full subtree starting at vertex $v$, with $\sigma^{(.)}\tau=\tau$. 
There exists a natural measure $\mu$ on the collection of trees, induced by $\mathbf{p}$. Let $[\tau]_{k}$ be the collection of trees $\kappa$ such that $\tau(v)=\kappa(v)$ for all nodes $v$ in levels up to $k$. Similarly to the $1$-variable setting we refer to this as a \emph{cylinder} and note that it generates the topology of $\mathcal{T}$. The measure $\mu$ is then the unique measure on $\mathcal{T}$ such that $\mu([\tau]_{k})=p_{1}^{\rho(1,k,\tau)}p_{2}^{\rho(2,k,\tau)}\hdots p_{n}^{\rho(n,k,\tau)}$, where $\rho(i,k,\tau)$ is the number of choices of letter $i\in\Lambda$ for all nodes up to level $k$ in realisation $\tau$.

We note that in this section we relax the requirement that every IFS $\mathbb{I}_{i}$ must contain at least one map, with a single IFS consisting of two maps. We now allow an IFS to have no maps, i.e.\ $W_{i}=\varnothing$ with positive probability, but we require a non-extinction condition.
\begin{defn}\label{defn:nonextinct}
We call the RIFS $(\mathbb{L},\vec{\pi})$ \emph{non-extinguishing} if 
\[
\sum_{i\in\Lambda}p_{i}\#\mathbb{I}_{i}>1.
\]
\end{defn}
This implies that there exists positive probability that the associated attractor (defined below) is non-empty. We will later state results `conditioned on non-extinction' by which we mean `with respect to the (normalised) measure $\mu$ restricted on the set of non-extinction'.

Allowing for extinction we  have to extend the definition of the modified singular value function.

\begin{defn}
Let $e\in\beth^{\odot}$, we define the \emph{upper (random) modified singular value function} as
\[
\overline{\psi}_\tau^s(e)=\begin{cases}
\alpha_{M}(e)^{\overline{s}(e,F_{\tau})}\alpha_{m}(e)^{s-\overline{s}(e,F_{\tau})}, & \text{ if }e\neq\varnothing,
\\
0, & \text{ otherwise.}
\end{cases}
\]
Again, $\underline{\psi}_\omega^s(e)$ and ${\psi}_\omega^s(e)$ are defined analogously.
\end{defn}

Recall that $e_{i}^{j}$ is the letter representing the map $f_{i}^{j}\in\mathbb{I}_{i}$. For each full tree $\tau$ that is randomly labelled by entries in $\Lambda$, we associate another rooted labelled $k_{s}$-ary tree $\mathbf{T}_{\tau}$, where each node is labelled by an arrangement of words that describes the `coding' of the associated cylinder.

\begin{defn}\label{infinityDef1}
Let $\mathbf{T}_{\tau}$ be a labelled tree, 
we write $\mathbf{T}_{\tau}(v)$ for the label of node $v$ of the tree $\mathbf{T}_{\tau}$.
The \emph{coding tree} $\mathbf{T}_{\tau}$ is then defined inductively:
\[
\mathbf{T}_{\tau}((.))=\epsilon_{0}\text{ and }
\mathbf{T}_{\tau}(v)=\mathbf{T}_{\tau}((v_{1},\dots,v_{k}))=\mathbf{T}_{\tau}((v_{1},\dots,v_{k-1}))\odot e_{\tau(v_{k-1})}^{v_{k}}
\]
for $1\leq v_{k}\leq \#\mathbb{I}_{\tau(v_{k-1})}$ and $e_{\tau(v_{k-1})}^{v_{k}}=\varnothing$ otherwise. This `deletes' this subbranch as $\varnothing$ annihilates under multiplication.

We refer to the arrangement of all labels at the $k$-th level by 
\[
\mathbf{T}_{\tau}^{k}=\bigsqcup_{v_{1},\dots,v_{k}\in\Lambda}\mathbf{T}_{\tau}((v_{1},\dots,v_{k})).
\]
\end{defn}
We remark that the resulting tree will almost surely, conditioned on non-extinction, have an exponentially increasing number of vertices at level $k$ as $k$ increases.
We can now define the random recursive, or $\infty$-variable, box-like self-affine carpet.
\begin{defn}\label{infinityDef2}
Let $({\mathbb{I}},\mathbf{p})$ be a box-like self-affine RIFS and $\tau\in\mathcal{T}$. The \emph{$\infty$-variable box-like self-affine carpet} $F_{\tau}$ is the compact set satisfying
\[
F_{\tau}=\bigcap_{k=1}^{\infty}f(\mathbf{T}_{\tau}^{k},\Delta).
\]
\end{defn}
We note that setting up the RIFS appropriately this models reduces to self-affine fractal percolation.

We write $s^{x}$ and $s^{y}$ for the almost sure box-counting dimension of the projections of $F_{\tau}$ onto the horizontal and vertical axes. In this case the projections are $\infty$-variable RIFSs or random graph directed systems (RGDSs) in the sense of \cite{Troscheit15} (see Definition~\ref{RGDS} below) and in the non-separated case $s^{x}=s^{y}$ almost surely.\index{random graph directed system}  

\begin{theo}\label{mainInfinityResult}
Let $({\mathbb{I}},\mathbf{p})$ be a box-like self-affine RIFS that satisfies the UORC and is non-extinguishing.
Let $F_{\tau}$ be the associated $\infty$-variable random self-affine box-like carpet.
The box-counting dimension of $F_{\tau}$, conditioned on non-extinction, is almost surely given by the unique $s_{B}$ satisfying
\begin{equation}\label{subadditiveExpecationEqui}
\E\left(\left(\sum_{e\in\mathbf{T}_{\tau}^{k}}\overline{\psi}_{\tau}^{s_{B}}(e)\right)^{1/k}\right)\to1\text{ as }k\to\infty.
\end{equation}
\end{theo}
\begin{cor}
Let $({\mathbb{I}},\mathbf{p})$ be a box-like self-affine RIFS that satisfies the UORC and is non-extinguishing.
Let $F_{\tau}$ be the associated $\infty$-variable random self-affine box-like carpet.
If the modified singular value function is additive, e.g.\  if $\alpha_{M}(e\odot g)=\alpha_{M}(e)\alpha_{M}(g)$, we have, conditioned on non-extinction, $\dim_{B}F_{\omega}=s_{B}$ almost surely, where
\begin{equation}\label{additiveinfty}
\E\left(\sum_{e\in\mathbb{I}_{i}}\overline{\psi}_{\tau}^{s_{B}}(e)\right)=1.
\end{equation}
\end{cor}
Similarly to the $1$-variable RIFSs we can get a dimension drop for $\infty$-variable carpets.
\begin{rem}
Let $(\mathbb{I},\mathbf{p})$ be a box-like self-affine RIFS that satisfies the UORC but does not have an additive modified singular value function.
Let $F_{\tau}$ be the $\infty$-variable attractor associated with $(\mathbb{I},\mathbf{p})$. Then the almost sure box-counting dimension of the attractor can drop below the least box-counting dimension of the individual attractors, that is there exists $\mathbb{I}$ such that, almost surely,
\[
\dim_{B}F_{\tau}<\min_{i\in\Lambda}\dim_{B}F_{i,i,i,\dots}\,.
\]
See Example~\ref{weirdlyConvergingExample} in Section~\ref{sec:examples}.
\end{rem}

Finally, we remark that the the box-counting dimension of $\infty$-variable attractors is always an upper bound to the box-counting dimension of $1$-variable attractors. When $\overline\psi$ is additive this can be easily seen by Jensen's inequality, noting that (\ref{shortenedEq}) is a geometric and (\ref{additiveinfty}) an arithmetic average.

\section{Examples}\label{sec:examples}
We now use our results to compute the box-counting dimension of some simple random self-affine sets.
\subsection{Example}\label{simpleExtensionExample}
\begin{figure}[htbp]
\begin{center}
\hfill\setlength{\fboxsep}{-.5pt}
\fbox{\includegraphics[width=0.30\textwidth]{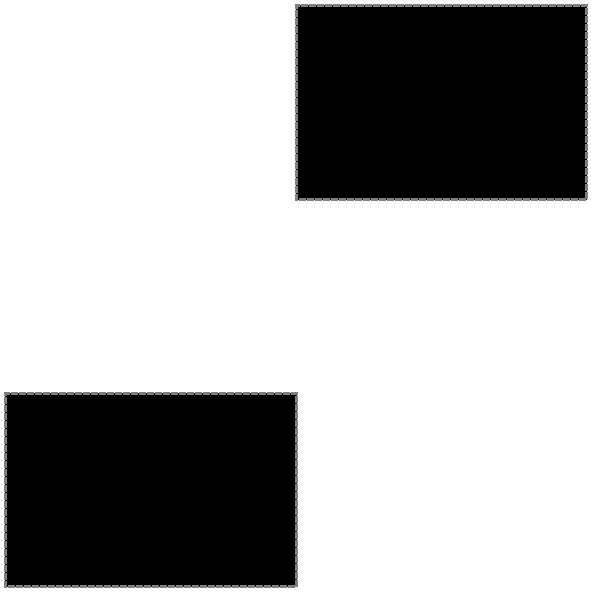}}\hfill
\fbox{\includegraphics[width=0.30\textwidth]{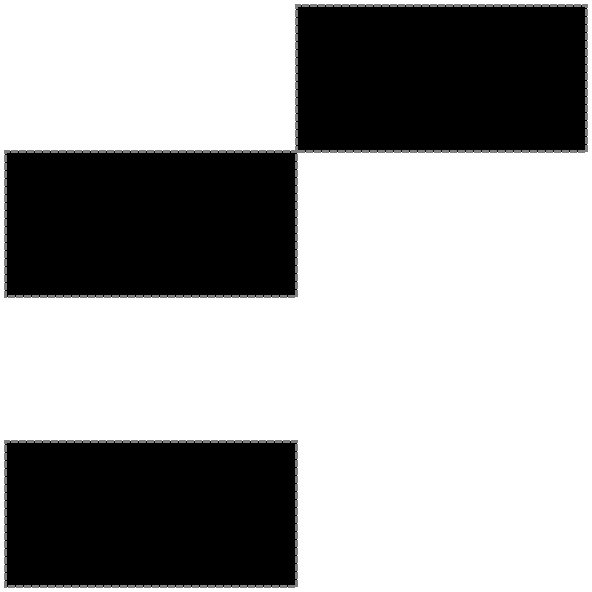}}\hfill\mbox{}
\caption{The two iterated functions systems used in Example~\ref{simpleExtensionExample}. The IFS $\mathbb{I}_{1}$ is to the left and $\mathbb{I}_{2}$ is to the right.}
\label{BMCarpetExample}
\end{center}
\end{figure}
\begin{figure}[htbp]
\begin{center}\setlength{\fboxsep}{-.5pt}
\fbox{\includegraphics[width=0.30\textwidth]{BMCarpetEx1}}\hfill
\fbox{\includegraphics[width=0.30\textwidth]{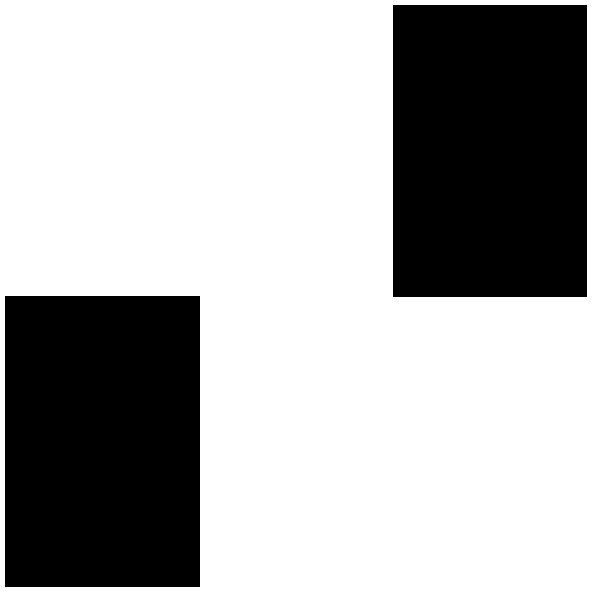}}\hfill
\fbox{\includegraphics[width=0.30\textwidth]{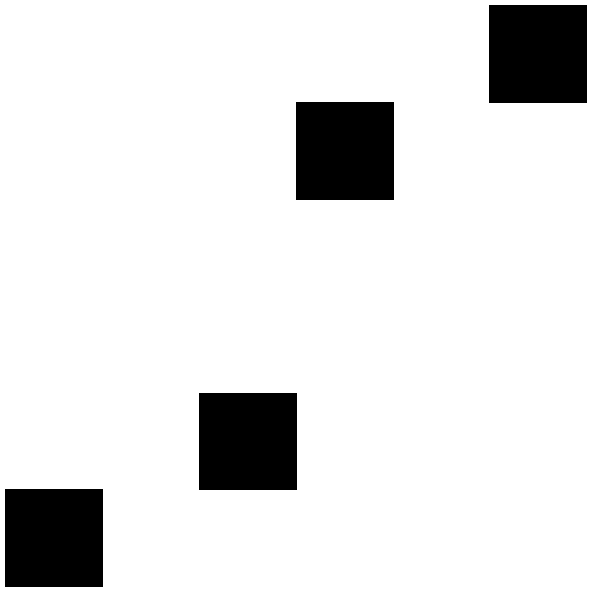}}
\caption{The two iterated functions systems used in Example~\ref{weirdlyConvergingExample} (left and middle) and the second iteration, choosing $\mathbb{I}_{1}$ and $\mathbb{I}_{2}$ in the construction.}
\label{BMCarpetExample2}
\end{center}
\end{figure}
Let $\mathbb{I}_{1}$ be the IFS for a Bedford-McMullen carpet $F_{1,1,\dots}$ with subdivision $n_{1}=2$ and $m_{1}=3$, consisting of two maps; one maps the unit square to a rectangle in the left column and one maps the unit square into a rectangle of the right column.
Take $\mathbb{I}_{2}$ to be the IFS for a Bedford-McMullen carpet $F_{2,2,\dots}$ with subdivision $n_{2}=2$ and $m_{2}=4$, consisting of three maps, two mapping into the left column and one mapping into the right column, see Figure~\ref{BMCarpetExample}. Note that for both IFSs the box-counting dimension of the projection onto the horizontal axis is $1$ and consider the $1$-variable box-like self-affine carpet associated with $\mathbb{I}=\{\mathbb{I}_{1},\mathbb{I}_{2}\}$ and $\mathbf{p}=\{1/2,1/2\}$.

The modified singular value function takes the value $\psi_{1}^{t}=(1/2)(1/3)^{t-1}$ for all elements $e\in\mathbb{I}_{1}$ and $\psi_{2}^{t}=(1/2)(1/4)^{t-1}$ for all $e\in\mathbb{I}_{2}$. Solving $\sum_{e\in\mathbb{I}_{1}}\psi_{1}^{t}=1$ and $\sum_{e\in\mathbb{I}_{2}}\psi_{2}^{t}=1$ for $t$ we get $\dim_{B}F_{1,1,\dots}=1$ and $\dim_{B}F_{2,2,\dots}=\log6/\log4$. However, as $\mu$ is the $(1/2,1/2)$-Bernoulli measure, substituting into~(\ref{shortenedEq}) and solving for $s_{B}$ we get that, almost surely, $\dim_{B}F_{\omega}=\log18/\log12$ and since $\log18/\log12>(1+\log6/\log4)/2$ equation~(\ref{expectationOfIndivid}) fails even in the simple setting of Bedford-McMullen carpets with mixed subdivisions.

\subsection{Example}\label{weirdlyConvergingExample}
Let $\mathbb{I}_{1}$ be the IFS for a Bedford-McMullen carpet as in the previous example and let $\mathbb{I}_{2}$ be another Bedford-McMullen carpet but with major contraction in the horizontal, see Figure~\ref{BMCarpetExample2}. Note that the periodic word $\widetilde\omega=(1,2,1,2,\dots)$ describes a self-similar set with Hausdorff and box-counting dimension $\log 4/\log 6$. It is easy to check that the individual Bedford-McMullen carpets have box-counting dimension $1$. Again, $\mathbb{I}=\{\mathbb{I}_{1},\mathbb{I}_{2}\}$ is of separated type, but has a non-additive   modified singular value function. We can calculate the box-counting dimension for the $1$-variable model explicitly, assuming the $(1/2,1/2)$-Bernoulli measure. First note that both projections of the attractor have, by symmetry, the same dimension $s_{p}$, which is the unique real satisfying 
\[
\exp\E\log\sum_{e\in W_{\omega_{1}}}\Lip(\pi_{x}f(e,.))^{s_{p}}=(2(1/2)^{s_{p}}2(1/3)^{s_{p}})^{1/2}=1, 
\]
where $\Lip$ refers to the Lipschitz constant of the similarity $f_e(x)=f(e,x)$ with respect to $x$. So, almost surely, $s_{p}=\log 4/\log 6$.
The box-counting dimension of the carpet then becomes the unique $t$ satisfying
\[
\left(2^{n}\left(2^{-k(n)}3^{k(n)-n}\right)^{\log 4/\log 6}\left(2^{k(n)-n}3^{-k(n)}\right)^{t-\log4/\log6}\right)^{1/n}\to 1\;\;\text{ a.s.\ as }n\to\infty,
\]
where $k(n)$ is maximum number of $1$s or $2$s in a randomly picked string $\{1,2\}^{n}$. Since $k(n)/n\to1/2$ a.s.\ we deduce $t=\log4/\log6$ and, almost surely, the box-counting dimension of the $1$-variable carpet agrees with that of the periodic word $\widetilde\omega$ and is strictly less than the box-dimensions of the individual attractors.

Taking the same $\mathbb{I}_{1}$, $\mathbb{I}_{2}$ and $\mathbf{p}=\{1/2,1/2\}$ but the $\infty$-variable construction, we can calculate the almost sure box-counting dimension of the projection of the carpet to be the unique $s_{p}$ satisfying 
\[
\E\sum_{e\in W_{\omega_{1}}}\Lip(\pi_{x}f(e,K))^{s_{p}}=2^{-s_{p}}+3^{-s_{p}}=1.
\]
For tree levels that are odd, the maximal singular value $\alpha_{M}$ cannot equal the lower singular value $\alpha_{m}$. We can explicitly state the expectation of the $2k+1$ level sum of the modified singular value function for $k\in\N_{0}$ by noting that for a binary tree of level $2k+1$ with two choices of labels per node there are $2^{1}2^{2}2^{4}\hdots 2^{2^{(2k+1)-1}}=2^{1+2+\hdots+2^{2k}}=2^{2^{2k+1}-1}$ choices of trees and thus $2^{2^{2k}}2^{2^{2k+1}-1}=2^{2^{2k+1}+2^{2k}-1}=2^{3 \cdot 2^{2k}-1}$ equally likely paths.
These paths correspond to all values of $2\alpha_{M}^{s_{B}}\alpha_{m}^{t-s_{B}}$ we want to sum up. Notice that $\alpha_{M}=2^{-i}3^{-j}$ and $\alpha_{m}=2^{-j}3^{-i}$ for some $i\geq j$ at level $k=i+j$. Thus for subtrees the new singular values can only be $2^{-i-1}3^{-j}$ if $i=j$, and $2^{-i-1}3^{-j}$ or $2^{-i}3^{-j-1}$ otherwise. Therefore the number of choices for $i$ at level $2k+1$ must be
\[
\binom{2k+1}{i} 2^{2^{2k+1}-1}
\]
for $k+1\leq i\leq 2k+1$. 
We can now state the expectation at level $2k+1$,
\begin{align*}
\E\sum_{e\in\mathbf{T}_{\tau}^{2k+1}}\overline\psi_{\tau}^{t}(e)&=\frac{1}{2^{2^{2k+1}-1}}\sum_{i=k+1}^{2k+1}2^{2^{2k+1}-1}\binom{2k+1}{i}\\
&\hspace{3cm}\cdot2(2^{-i}3^{-(2k+1-i)})^{s_{p}}(2^{-(2k+1-i)}3^{-i})^{t-s_{p}}\\ & \\
&=2\sum_{i=k+1}^{2k+1}\binom{2k+1}{i}\left(\frac{3}{2}\right)^{s_{p}2i}\left(\frac{3}{2}\right)^{-s_{p}(2k+1)}2^{-t(2k+1)}\left(\frac{2}{3}\right)^{it}\\
&=2^{1-t(2k+1)}\left(\frac{2}{3}\right)^{s_{p}(2k+1)}\sum_{i=k+1}^{2k+1}\binom{2k+1}{i}\left(\frac{3}{2}\right)^{i(2s_{p}-t)}\\
&\leq2^{1-t(2k+1)}\left(\frac{2}{3}\right)^{s_{p}(2k+1)}\left(1+\left(\frac{3}{2}\right)^{(2s_{p}-t)}\right)^{2k+1}.
\end{align*}
Let $\epsilon>0$ and set $t=s_{p}+\epsilon$, then
\begin{align*}
&\leq2^{1-\epsilon(2k+1)}3^{-s_{p}(2k+1)}\left(1+\left(\frac{3}{2}\right)^{s_{p}}\left(\frac{3}{2}\right)^{-\epsilon}\right)^{2k+1}\\
&=2^{1-\epsilon(2k+1)}3^{-s_{p}(2k+1)}\left(1+(3^{-s_{p}}-1)\left(\frac{3}{2}\right)^{-\epsilon}\right)^{2k+1}\\
&\to 0\;\;\;\;\text{ as }k\to\infty.
\end{align*}
Therefore the box-counting dimension $s_{B}$ of the attractor satisfies $s_{p}\leq s_{B}\leq s_{p}+\epsilon$ and so $s_{B}=s_{p}$. Again this is a dimension drop when compared to the original attractors.

\section{Proofs}\label{proofs}
This section is divided into two parts; we prove the results of Section~\ref{1varSect} in the first, followed by the proofs for Section~\ref{inftyVarRes}. First, however, we recall the definition of random graph directed systems introduced in Troscheit~\cite{Troscheit15} and then state and prove Lemmas \ref{lma:constructionOfProjections} and \ref{lma:almostsureprojections}, which apply to both $1$-variable and $\infty$-variable constructions.

\begin{defn}\label{RGDS}
Let $\mathbf{\Gamma}=\{\Gamma(1),\dots,\Gamma(N)\}$ be a finite collection of directed graphs with edge sets $E(i)$, $1\leq i\leq N$ and common vertex set $V=\{1,\dots,M\}$. Associate to each graph a probability $p_{i}>0$ such that $\sum p_{i}=1$. We represent each edge $\mathbf{e}\in E(i)$ by a unique prime arrangement of words $e\in\mathcal{G}^{E}$ that codes an associated contraction $f_{\mathbf{e}}$.
For $v,w\in V$ write $\tensor*[_{v}]{E}{_{w}}(i)=e_{1}\sqcup\dots\sqcup e_{n}$, where the edges associated with the prime arrangements $e_{k}\in E(i)$ have initial vertex $v$ and terminal vertex $w$. We set $\tensor*[_{v}]{E}{_{w}}(i)=\varnothing$ if no such edge exists.
Let
\[
\mathbf{E}(i)=\begin{pmatrix}
\tensor*[_{1}]{E}{_{1}}(i) & \tensor*[_{1}]{E}{_{2}}(i) & \dots & \tensor*[_{1}]{E}{_{M}}(i)\\
\tensor*[_{2}]{E}{_{1}}(i) & \ddots &  & \tensor*[_{1}]{E}{_{M}}(i)\\
\vdots&&\ddots&\vdots\\
\tensor*[_{M}]{E}{_{1}}(i) & \tensor*[_{M}]{E}{_{2}}(i) & \dots & \tensor*[_{M}]{E}{_{M}}(i)
\end{pmatrix}
\]
and $\mathds{1}_{v}$ be a vector of length $M$ such that $(\mathds{1}_{v})_{k}=\epsilon_{0}$ if $k=v$ and $(\mathds{1}_{v})_{k}=\varnothing$ otherwise. 
\end{defn}
Matrix multiplication $\times$ and addition $\sqcup$ for such $n$ by $n$ matrices $\mathbf{M}$ and $\mathbf{N}$, and vectors $\mathbf{v}$ of dimension $n$ are defined in the natural way:
\begin{equation}\label{matrixMult}
(\mathbf{M}\times\mathbf{N})_{i,j}=\bigsqcup_{k=1}^{n}(\mathbf{M}_{i,k}\odot\mathbf{N}_{k,j}),\hspace{0.5cm}
(\mathbf{M}\sqcup\mathbf{N})_{i,j}=\mathbf{M}_{i,j}\sqcup\mathbf{N}_{i,j},
\end{equation}
\[
(\mathbf{v}\times\mathbf{M})_{i}=\bigsqcup_{k=1}^{n}({v}_{k}\odot\mathbf{M}_{k,i}).
\]
\begin{defn}\label{defn:rgdsdef}
Let
\[
\mathbf{E}_{v}^{k}(\omega)=\bigsqcup_{1\leq j\leq M}\left(\mathds{1}_{v}\mathbf{E}(\omega_{1})\mathbf{E}(\omega_{2})\dots\mathbf{E}(\omega_{k})\right)_{j}
\]
be the \emph{$k$-level sets} for vertex $i\in\{1,\dots,N\}$. The \emph{$1$-variable random graph directed self similar set} $K_{v}(\omega)$ associated to $v\in V$ and realisation $\omega\in\Lambda^{\N}$ is given by
\[
K_{v}(\omega)=\bigcap_{k>0}f(\mathbf{E}_{v}^{k}(\omega),\Delta).
\]
The $\infty$-variable random graph directed system is defined analogously to $\infty$-variable RIFS (see Definitions~\ref{infinityDef1} and \ref{infinityDef2}), replacing $e_{\tau(v_{k-1})}^{v_{k}}$ by an appropriately chosen matrix $\mathbf{E}(i)$ and taking the sum of arrangements over the $v$-th column.
\end{defn}

\begin{defn}\label{graphDefs}
Let $\mathbf{\Gamma}=\{\Gamma_{i}\}_{i\in\Lambda}$ be a finite collection of graphs, sharing the same vertex set $V$. 
\begin{enumerate}[label=\arabic{section}.\arabic{theo}.\alph*]
\item\label{nontrivialgraph} We say that the collection $\mathbf{\Gamma}$ is a \emph{non-trivial collection of graphs} if for every $i\in\Lambda$ and $v\in V$ we have 
\[
\bigsqcup_{w\in V}\tensor*[_{v}]{E}{_{w}}(i)\neq\varnothing.
\]
Furthermore we require that there exist $e_{1}\in E(i)$ and $e_{2}\in E(j)$ such that $f_{e_{1}}\neq f_{e_{2}}$.

\item\label{strongConnected} If for every $v,w\in V$ there exists $\omega^{v,w}\in\Lambda^{*}$ such that $\tensor[_{v}]{E}{_{w}}(\omega^{v,w})\neq\varnothing$ and $\mu([\omega^{v,w}])>0$, we call $\mathbf{\Gamma}$ \emph{stochastically strongly connected}.

\item\label{contractingSSRGDS} We call the \emph{Random Graph Directed System (RGDS)} associated with $\mathbf{\Gamma}$ a \emph{contracting self-similar RGDS} if for every $e\in E(i)$, $f_{e}$ is a contracting similitude.
\end{enumerate}
\end{defn}

\begin{cond}\label{replacementCondition}
Let $\mathbf{\Gamma}=\{\Gamma_{i}\}_{i\in\Lambda}$ be a finite collection of graphs, sharing the same vertex set $V$.
We say that the collection $\mathbf{\Gamma}$ is a \emph{non-trivial surviving collection of graphs} if for every $v\in V$ we have $\E(\#\{e\in\tensor*[_{v}]{E(\omega_{1})}{_{w}}\mid w\in V)>1$.
Furthermore we require that there exist $e_{1}\in E(i)$ and $e_{2}\in E(j)$ such that $f_{e_{1}}\neq f_{e_{2}}$.
\end{cond}

Condition~\ref{replacementCondition} is similar to a RIFS being non-extinguishing (Definition~\ref{defn:nonextinct}), guaranteeing the existence of a positive probability that $K_{v}(\omega)\neq \varnothing$.

\begin{lma}\label{lma:constructionOfProjections}
Let $(\mathbb{I},\mathbf{p})$ be a box-like self-affine RIFS with associated $1$-variable ($\infty$-variable) carpet $F_{\omega}$ ($F_{\tau}$). The projections $\pi_x(F_\omega)$ and $\pi_y(F_\omega)$, and  $\pi_x(F_\tau)$ and $\pi_y(F_\tau)$ are, in the separated case, random self-similar and, in the non-separated case, random graph-directed self-similar sets (RGDS) as above. 
The $1$-variable RGDS satisfies all conditions in Definition~\ref{graphDefs} (assuming every IFS has at least one map, with at least one IFS having two map) and the $\infty$-variable RGDS (assuming it is non-extinguishing) satisfies all conditions in Definition~\ref{graphDefs} with Condition~\ref{replacementCondition} replacing \ref{nontrivialgraph}.
\end{lma}
\begin{proof}
We prove the $1$-variable case where $(\mathbb{I},\mathbf{p})$ is a $1$-variable box-like self-affine RIFS.
Assume that $\mathbb{I}$ is separated; without loss of generality (considering all possible concatenations of two $f_{i}^{j}$ if necessary) for every $i\in\Lambda$ and $j\in\mathbb{I}_{i}$ the matrix $\mathbf{Q}_{i}^{j}$ does not have off-diagonal entries and for some $\widehat a_{i}^{j},\widehat b_{i}^{j}\in\R\setminus\{0\}$ each map $f_{i}^{j}$ can be rewritten as
\begin{equation}\label{matrixFormSep}
f_{i}^{j}(\mathbf{x})=\begin{pmatrix}
\widehat a_{i}^{j}	&	0\\
0			&	\widehat b_{i}^{j}
\end{pmatrix}
\mathbf{x}+\begin{pmatrix} u_{i}^{j} \\ v_{i}^{j}\end{pmatrix}.
\end{equation}
We define the two induced maps $\Pi_{x}f_{i}^{j}:\R\to\R$ and $\Pi_{y}f_{i}^{j}:\R\to\R$ by 
\[
\Pi_{x}f_{i}^{j}(z)=\pi_{x}\circ f_{i}^{j}\circ\pi_{x}^{-1}(z)=\widehat a_{i}^{j} z +u_{i}^{j}\;\;\;\text{ and }\;\;\; 
\Pi_{y}f_{i}^{j}(z)=\pi_{y}\circ f_{i}^{j}\circ\pi_{y}^{-1}(z)=\widehat b_{i}^{j} z +v_{i}^{j}.
\]
For every $e\in\mathbf{C}_{\omega}^{k}$ 
\begin{align*}
\Pi_{x}f_{e}&=\pi_{x}\circ f_{e_{1}}\circ f_{e_{2}}\circ\dots\circ f_{e_{\lvert e\rvert}}\circ \pi_{x}^{-1}\\
&=\pi_{x}\circ f_{e_{1}}\circ\pi_{x}^{-1}\circ\pi_{x}\circ f_{e_{2}}\circ\dots\pi_{x}\circ f_{e_{\lvert e\rvert}}\circ \pi_{x}^{-1}\\
&=(\Pi_{x}f_{e_{1}})\circ(\Pi_{x}f_{e_{2}})\circ\dots\circ (\Pi_{x}f_{e_{\lvert e\rvert}}),
\end{align*}
where $f_{e}(x)=f(e,x)$.
So the attractor $K_{x,\omega}$ of the iterated function system $\mathbb{I}^{x}=\{\mathbb{I}_{i}^{x}\}_{i\in\Lambda}$, with $\mathbb{I}_{i}^{x}=\{\Pi_{x}f_{i}^{j}\}_{j\in\mathcal{I}_{i}}$ satisfies $K_{x,\omega}=\pi_{x}F_{\omega}$. 
A similar argument holds for the projection onto the vertical axis and thus the projections of $F_{\omega}$ onto the $x$ and $y$ axes are the attractors of the RIFSs $\mathbb{I}^{x}$ and $\mathbb{I}^{y}$. Finally note that the projections are similitudes and hence the projections of $F_{\omega}$ onto the horizontal and vertical axes are self-similar RIFS.

The argument for the graph directed construction in the non-separated case is similar. For each $i\in\Lambda$ we define a graph $\Gamma_{i}$, where all graphs $\mathbf{\Gamma}=\{\Gamma_{1},\dots,\Gamma_{N}\}$ share the same vertex set $V=\{V_{h},V_{v}\}$. We define a new set of prime arrangements 
$\mathcal{G}^{*E}=\{e_{i,j}^{z}\mid z\in\{x,y\}\}$ for the RGDS, where for every $e_{i}^{j}$ we obtain two new unique words: $e_{i,j}^{x}$ and $e_{i,j}^{y}$. Let $f_{i}^{j}$ be of the same form as~(\ref{matrixFormSep}), we define 
\[
f(e_{i,j}^{x},z)=\pi_{x}\circ f_{i}^{j}\circ\pi_{x}^{-1}(z)=\widehat{a}_{i}^{j}z+u_{i}^{j}\;\;\;\text{and}\;\;\;
f(e_{i,j}^{y},z)=\pi_{y}\circ f_{i}^{j}\circ\pi_{y}^{-1}(z)=\widehat{b}_{i}^{j}z+v_{i}^{j}.
\]
For $f_{i}^{j}$ that map vertical lines to horizontal and vice versa, i.e. are of form
\begin{equation}\label{matrixFormNonSep}
f_{i}^{j}(\mathbf{x})=\begin{pmatrix}
0				&	\widehat a_{i}^{j}\\
\widehat b_{i}^{j}	&	0
\end{pmatrix}
\mathbf{x}+\begin{pmatrix} u_{i}^{j} \\ v_{i}^{j}\end{pmatrix},
\end{equation}
we define 
\[
f(e_{i,j}^{y},z)=\pi_{y}\circ f_{i}^{j}\circ\pi_{x}^{-1}(z)=\widehat{b}_{i}^{j}z+v_{i}^{j}\;\;\;\text{and}\;\;\;
f(e_{i,j}^{x},z)=\pi_{x}\circ f_{i}^{j}\circ\pi_{y}^{-1}(z)=\widehat{a}_{i}^{j}z+u_{i}^{j}.
\]

Fix $i\in\Lambda$. We define, for all $v,w\in V$ the edge set $\tensor*[_{v}]{E}{_{w}}(i)$ by
\[
\tensor*[_{V_{h}}]{E}{_{V_{h}}}(i)=\bigsqcup\{e^{x}_{i,j} \mid f_{i}^{j}\text{ is of form (\ref{matrixFormSep})}\},
\]
\[
\tensor*[_{V_{v}}]{E}{_{V_{v}}}(i)=\bigsqcup\{e^{y}_{i,j} \mid f_{i}^{j}\text{ is of form (\ref{matrixFormSep})}\},
\]
\[
\tensor*[_{V_{h}}]{E}{_{V_{v}}}(i)=\bigsqcup\{e^{x}_{i,j} \mid f_{i}^{j}\text{ is of form (\ref{matrixFormNonSep})}\}, 
\]
\[
\tensor*[_{V_{v}}]{E}{_{V_{h}}}(i)=\bigsqcup\{e^{y}_{i,j} \mid f_{i}^{j}\text{ is of form (\ref{matrixFormNonSep})}\}.
\]

It remains to check that $\pi_{x}F_{\omega}=K_{V_{h}}(\omega)$ and $\pi_{y}F_{\omega}=K_{V_{v}}(\omega)$, for which it is sufficient to show that, given any finite arrangement of words in $e\in\mathbf{C}_{\omega}^{*}$,
\begin{equation}\label{projIsTheSame}
\pi_{x}f(e,\Delta)=f(e^{x},[0,1])\;\;\text{and}\;\; \pi_{y}f(e,\Delta)=f(e^{y},[0,1])
\end{equation}
is satisfied, where $e^{x}$ and $e^{y}$ are the two induced paths, starting at $V_{h}$ and $V_{v}$, respectively.

If $e=\epsilon_{0}$, then $e^{x}=e^{y}=\epsilon_{0}$ and if $e=\varnothing$, then $e^{x}=e^{y}=\varnothing$ and in both cases~(\ref{projIsTheSame}) holds trivially.
Let $w=w_{1}\odot\dots\odot w_{k}\in\mathbf{C}_{\omega}^{k}$ with $w_{1}=e_{i}^{j}$ and write $w^{x}$ and $w^{y}$ for the induced paths. Assume inductively that $\pi_{x}f(w_{2}\odot\dots\odot w_{k},\Delta)=f((w_{2}\odot\dots\odot w_{k})^{x},[0,1])$ and $\pi_{y}f(w_{2}\odot\dots\odot w_{k},\Delta)=f((w_{2}\odot\dots\odot w_{k})^{y},[0,1])$.
Consider the map $f_{i}^{j}$ and assume first it is of the form~(\ref{matrixFormSep}). The map $f(w,.)$ can be written as
\begin{align*}
\begin{pmatrix}
z_{1} \\
z_{2}
\end{pmatrix}
&\mapsto
\begin{pmatrix}
\widehat{a}_{i}^{j}\cdot\pi_{x}\circ f(w_{2}\odot\dots\odot w_{k},(z_{1},z_{2})^{\top})+u_{i}^{j}\\
\widehat{b}_{i}^{j}\cdot\pi_{y}\circ f(w_{2}\odot\dots\odot w_{k},(z_{1},z_{2})^{\top})+v_{i}^{j}
\end{pmatrix}\\
&=
\begin{pmatrix}
\widehat{a}_{i}^{j}\cdot f((w_{2}\odot\dots\odot w_{k})^{x},z_{1})+u_{i}^{j}\\
\widehat{b}_{i}^{j}\cdot f((w_{2}\odot\dots\odot w_{k})^{y},z_{2})+v_{i}^{j}
\end{pmatrix}\\
&=
\begin{pmatrix}
f(w^{x},z_{1})\\
f(w^{y},z_{2})
\end{pmatrix}.
\end{align*}
Analogously, if $f_{i}^{j}$ is of form~(\ref{matrixFormNonSep}),
\begin{align}
\begin{pmatrix}
z_{1} \\
z_{2}
\end{pmatrix}
&\mapsto
\begin{pmatrix}
\widehat{b}_{i}^{j}\cdot\pi_{y}\circ f(w_{2}\odot\dots\odot w_{k},(z_{1},z_{2})^{\top})+v_{i}^{j}\\
\widehat{a}_{i}^{j}\cdot\pi_{x}\circ f(w_{2}\odot\dots\odot w_{k},(z_{1},z_{2})^{\top})+u_{i}^{j}
\end{pmatrix}\nonumber\\
&=
\begin{pmatrix}
\widehat{b}_{i}^{j}\cdot f((w_{2}\odot\dots\odot w_{k})^{y},z_{1})+v_{i}^{j}\\
\widehat{a}_{i}^{j}\cdot f((w_{2}\odot\dots\odot w_{k})^{x},z_{2})+u_{i}^{j}
\end{pmatrix}\nonumber\\
&=
\begin{pmatrix}
f(w^{x},z_{1})\\
f(w^{y},z_{2})
\end{pmatrix},\label{graphSignChange}
\end{align}
where we have used that $w_{1}^{x}\in\tensor*[_{V_{h}}]{E}{_{V_{v}}}(\omega_{1})$ and $w_{1}^{y}\in\tensor*[_{V_{v}}]{E}{_{V_{h}}}(\omega_{1})$ in~(\ref{graphSignChange}). Therefore, by induction on word length, $\pi_{x}F_{\omega}^{k}=K^{k}_{V_{h}}(\omega)$ and $\pi_{y}F_{\omega}^{k}=K^{k}_{V_{v}}(\omega)$.
We conclude that $\pi_{x}F_{\omega}=K_{V_{h}}(\omega)$ and $\pi_{y}F_{\omega}=K_{V_{v}}(\omega)$, where $K_{i}(\omega)$ is the $1$-variable random graph directed system of Definition~\ref{defn:rgdsdef}, as all maps are similitudes.

Finally we check the conditions of Definition~\ref{graphDefs}. Non-triviality arises from the fact that $p_{i}>0$ for all $i\in\Lambda$, each map $f_{i}^{j}$ induces exactly one map starting at each of the two vertices and that we assume at least two maps to be distinct. Lastly, the stochastically strongly connected condition is satisfied since at least one of the maps is separated and there exists at least one pair $(i,j)$ such that horizontal get mapped to vertical lines.

The result for the $\infty$-variable case is almost identical and left to the reader.
\end{proof}

The following lemma follows directly from Theorems 2.24 ($1$-variable) and 3.6 ($\infty$-variable) in Troscheit~\cite{Troscheit15}.

\begin{lma}\label{lma:almostsureprojections}
Let $(\mathbb{I},\mathbf{p})$ be a box-like self-affine RIFS s.t. every IFS has at least one map, with at least one IFS having two maps (it is non-extinguishing) and let $F_{\omega}$ ($F_{\tau}$) be the associated $1$-variable ($\infty$-variable carpet. 
Let $e\in\mathbf{C}_{\omega}^{*}$ ($e\in\mathbf{T}_{\tau}^{*}$) be the level sets of Definition~\ref{defn:1varcoding} (Definition~\ref{infinityDef1}). Then $\overline{s}(e,F_{\omega})=\underline{s}(e,F_{\omega})$ ($\overline{s}(e,F_{\tau})=\underline{s}(e,F_{\tau})$) is constant almost surely and coincides with $s^{x}$ or $s^{y}$, the almost sure box-counting dimension of the projection of $F_{\omega}$ ($F_{\tau}$) onto the horizontal or vertical axis, respectively. If $\mathbb{I}$ is non-separated, then additionally $s^{x}=s^{y}$ almost surely.
\end{lma}

Basic dimensional properties give that the box-counting dimension of a set $X$ is bounded above by 
\[
\overline\dim_{B}X\leq\overline \dim_{B} (\pi_{x}X\times \pi_{y}X)\leq\overline\dim_{B}(\pi_{x}X)+\overline\dim_{B}(\pi_{y}X).
\]
So, almost surely the box-counting dimension for $1$-variable and $\infty$-variable box-like carpets cannot exceed $s^{x}+s^{y}$.
In the proofs below we will however also consider the parameter $s$ for $s>s^{x}+s^{y}$ to show that our results are exactly the unique values $s_{B}$ such that (\ref{mainEq}) and (\ref{additiveinfty}) hold, rather than $\min\{s_{B},s^{x}+s^{y}\}$.
Note that this means that the modified singular value function must be subadditive, although not strictly so.

\subsection{Proofs for Section~\ref{1varSect}}

The modified singular value function is in certain cases either stochastically subadditive, additive or superadditive.
The proof shares many similarities with \cite[Lemma 2.1]{Fraser12}, although differs in some points because $s^{x}(F_{\omega})$ and $s^{y}(F_{\omega})$ do not surely coincide.

\begin{lma}\label{lma:stochAdditivity}
Let $e\in\mathbf{C}_{\omega}^{*}$ and $g\in\beth^{*}$ be such that $e\odot g\in\mathbf{C}_{\omega}^{*}$. Writing $l=\lvert e\rvert$ and $\gamma_{\omega}(e,g), \gamma_{\min}, \gamma_{\max}$ for some quantities that will arise in the proof but are almost surely equal to one, the following statements hold for all $k\geq0$
\begin{enumerate}

\item If $t\in[0,\overline{s}^{x}(F_{\omega})+\overline{s}^{y}(F_{\omega})]$ then
\begin{eqnarray}
\overline{\psi}_\omega^t(e\odot g)&\leq& \gamma_{\omega}(e,g)\overline{\psi}_\omega^t(e) \overline{\psi}_{\sigma^l \omega}^t(g),\nonumber
\\
\overline{\Psi}_\omega^{k+l}(s)&\leq& \gamma_{\max}^{(k+l)\lvert \overline{s}^{x}(F_{\omega})-\overline{s}^{y}(F_{\omega})\rvert}\Psi_\omega^{k}(s)\overline{\Psi}_{\sigma^k\omega}^{l}(s).\label{presSubMult}
\end{eqnarray}

\item If $t\geq \overline{s}^{x}(F_{\omega})+\overline{s}^{y}(F_{\omega})$ then 
\begin{eqnarray}
\overline{\psi}_\omega^t(e\odot g)&\geq& \gamma_{\omega}(e,g)\overline{\psi}_\omega^t(e) \overline{\psi}_{\sigma^l \omega}^t(g),\nonumber\\
\overline{\Psi}_\omega^{k+l}(s)&\geq& \gamma_{\min}^{(k+l)\lvert \overline{s}^{x}(F_{\omega})-\overline{s}^{y}(F_{\omega})\rvert}\overline{\Psi}_\omega^{k}(s)\overline{\Psi}_{\sigma^k\omega}^{l}(s).\nonumber
\end{eqnarray}
\end{enumerate}
An analogous result holds for $\underline\psi_{\omega}^{t}$ and $\underline\Psi_{\omega}^{k}$.
\end{lma}
\begin{proof}
We first prove the results concerning the modified singular value function and deal with the separated case, which implies that  $\alpha_{M}(e)$ can only take the values of $h(e)h(g)$ or $w(e)w(g)$, where $h(z)=\lvert\pi_{y}f(z,\Delta)\rvert$ and $w(z)=\lvert\pi_{x}f(z   ,\Delta)\rvert$ are the height and width of the rectangle $f(z,\Delta)$.
Without loss of generality we can assume that $h(e)\geq w(e)$ i.e.\ $\alpha_{M}(e)=h(e)$.
We therefore have the following cases to check 
\begin{enumerate}[label=(\roman*)]
\item\label{case1List1} $\alpha_{M}(g)=h(g)$  and thus $\alpha_{M}(e\odot g)=h(e)h(g)$,

\item\label{case2List1} $\alpha_{M}(g)=w(g)$  and $\alpha_{M}(e\odot g)=h(e)h(g)$,

\item\label{case3List1} $\alpha_{M}(g)=w(g)$ and $\alpha_{M}(e\odot g)=w(e)w(g)$.
\end{enumerate}
In the separated case we define  $\gamma_{\min}=\gamma_{\max}=\gamma_{\omega}(e,g)=1$
and we shall now treat each of the cases above separately.

\ref{case1List1} We have
\begin{align}
\frac{\overline{\psi}_\omega^t(e\odot g)}{\overline{\psi}_\omega^t(e)\overline{\psi}_{\sigma^l\omega}^t(g)}
&=
\frac{(h(e)h(g))^{\overline{s}(e\odot g,F_{\omega})}(w(e)w(g)))^{t-\overline{s}(e\odot g,F_{\omega})}}
{h(e)^{\overline{s}(e,F_{\omega})}w(e)^{t-\overline{s}(e,F_{\omega})}h(g)^{\overline{s}(g,F_{\sigma^{l}\omega})}b(g)^{t-\overline{s}(g,F_{\sigma^{l}\omega})}},
\nonumber\\
&=\frac{(h(e)h(g))^{\overline{s}^{x}(F_{\omega})}(w(e)w(g))^{t-\overline{s}^{x}(F_{\omega})}}
{h(e)^{\overline{s}^{x}(F_{\omega})}w(e)^{t-\overline{s}^{x}(F_{\omega})}h(e)^{\overline{s}^{x}(F_{\omega})}w(g)^{t-\overline{s}^{x}(F_{\omega})}}
=1.\label{separatedRatioDecomposes}
\end{align}

\ref{case2List1} We have 
\begin{align*}
\frac{\overline{\psi}_\omega^t(e\odot g)}{\overline{\psi}_\omega^t(e)\overline{\psi}_{\sigma^l\omega}^t(g)}
&=\frac{(h(e)h(g))^{\overline{s}^{y}(F_{\omega})}(w(e)w(g))^{t-\overline{s}^{y}(F_{\omega})}}{h(e)^{\overline{s}^{y}(F_{\omega})}w(e)^{t-\overline{s}^{y}(F_{\omega})}w(g)^{\overline{s}^{x}(F_{\omega})}h(g)^{t-\overline{s}^{x}(F_{\omega})}}\\
&=\left(\frac{w(g)}{h(g)}\right)^{t-\overline{s}^{x}(F_{\omega})-\overline{s}^{y}(F_{\omega})}=r^{t-\overline{s}^{x}(F_{\omega})-\overline{s}^{y}(F_{\omega})}
\end{align*}
where $r>1$.

\ref{case3List1} We have
\begin{align*}
\frac{\overline{\psi}_\omega^t(e\odot g)}{\overline{\psi}_\omega^t(e)\overline{\psi}_{\sigma^l\omega}^t(g)}
&=\frac{(w(e)w(g))^{\overline{s}^{x}(F_{\omega})}(h(e)h(g))^{t-\overline{s}^{x}(F_{\omega})}}{h(e)^{\overline{s}^{y}(F_{\omega})}w(e)^{t-\overline{s}^{y}(F_{\omega})}w(g)^{\overline{s}^{y}(F_{\omega})}h(g)^{t-\overline{s}^{x}(F_{\omega})}}\\
&=\left(\frac{h(e)}{w(e)}\right)^{t-\overline{s}^{x}(F_{\omega})-\overline{s}^{y}(F_{\omega})}=r^{t-\overline{s}^{x}(F_{\omega})-\overline{s}^{y}(F_{\omega})}
\end{align*}
where $r>1$.

The required cases then follow by letting $t$ take the appropriate values.
\\ \\
For the non-separated case we have almost surely $s^{x}=s^{x}(F_{\omega})=s^{y}(F_{\omega})$. Let 
\[
\gamma_{\min}=\min_{\substack{i\in\Lambda\\h\in\mathbb{I}_{i}}}\frac{\alpha_{m}(h)}{\alpha_{M}(h)}\;\;\text{ and }\;\;\gamma_{\max}=\gamma_{\min}^{-1}.
\]
Note that $\alpha_{m}(e\odot g)\geq\alpha_{m}(e)\alpha_{m}(g)$. Equivalently $\alpha_{m}(e\odot g)=c(e,g)\alpha_{m}(e)\alpha_{m}(g)$ for some $c(e,g)\geq1$ and so we have for all $\omega\in\Lambda^{\N}$,
\begin{align}
\overline{\psi}_\omega^t(e \odot g)&=\alpha_{M}(e\odot g)^{\overline{s}(e\odot g,F_{\omega})}\alpha_{m}(e \odot g)^{t-\overline{s}(e\odot g,F_{\omega})}\nonumber\\
&=(\alpha_{M}(e\odot g)\alpha_{m}(e\odot g))^{\overline{s}(e\odot g,F_{\omega})}\alpha_{m}(e\odot g)^{t-2\overline{s}(e\odot g,F_{\omega})}\nonumber\\
&=(\alpha_{M}(e)\alpha_{M}(g)\alpha_{m}(e)\alpha_{m}(g))^{\overline{s}(e\odot g,F_{\omega})}\alpha_{m}(e\odot g)^{t-2\overline{s}(e\odot g,F_{\omega})}\nonumber\\
&=(\alpha_{M}(e)\alpha_{M}(g)\alpha_{m}(e)\alpha_{m}(g))^{\overline{s}(e\odot g,F_{\omega})}\nonumber\\
&\hspace{3cm}(\alpha_{m}(e)\alpha_{m}(g))^{t-2\overline{s}(e\odot g,F_{\omega})}c(e,g)^{t-2\overline{s}(e\odot g,F_{\omega})}\nonumber\\
&=c(e,g)^{t-2\overline{s}(e\odot g,F_{\omega})}\gamma_{\omega}(e,g)\overline{\psi}_\omega^t(e) \overline{\psi}_{\sigma^l\omega}^t(g),\label{eqn:individualbound}
\end{align}
for some $\gamma_{\omega}(e,g)\in\left[\gamma_{\min}^{\lvert e\odot g\rvert\cdot\lvert \overline{s}^{x}(F_{\omega})-\overline{s}^{y}(F_{\omega})\rvert},\gamma_{\max}^{\lvert e\odot g\rvert\cdot\lvert \overline{s}^{x}(F_{\omega})-\overline{s}^{y}(F_{\omega})\rvert}\right]$.
The required inequalities follow again by letting $t$ take appropriate values.
Note that, while $\gamma_{\omega}(e,g)=1$ holds for all $\omega\in\Lambda^{\N}$ in the separated case, in the non-separated case the equality holds only for $\omega\in\mathfrak{G}$, where $\mathfrak{G}$ is the full measure set on which ${s}^{x}$ and ${s}^{y}$ exist and coincide (guaranteed to exist by Lemma~\ref{lma:almostsureprojections} for RIFS of the non-separated type).
As it turns out this does not make a difference to the convergence of $P$.
\\ \\
We now move on to proving the inequalities involving $\overline{\Psi}_{\omega}^{k+l}(s)$. We have from (\ref{eqn:individualbound}) for $e\in\mathbf{C}_\omega^k$ and $g\in\mathbf{C}_{\sigma^k\omega}^{l}$ such that $e\odot g \in \mathbf{C}_\omega^{k+l}$ and $t\in [0,s_x+s_y]$,
\begin{align*}
\overline{\Psi}_\omega^{k+l}(t)&=\sum_{e\odot g \in \mathbf{C}_\omega^{k+l}}\overline{\psi}_\omega^t(e\odot g)\\
&=\sum_{e\in\mathbf{C}_\omega^k}\sum_{g\in\mathbf{C}_{\sigma^k\omega}^{l}}\overline{\psi}_\omega^t(e\odot g)
\leq\sum_{e\in\mathbf{C}_\omega^k}\sum_{g\in\mathbf{C}_{\sigma^k\omega}^{l}}\gamma_{\omega}(e,g)\overline{\psi}_\omega^t(e) \overline{\psi}_{\sigma^k\omega}^t(g)\\
&\leq\gamma_{\max}^{(k+l)\lvert \overline{s}^{x}(F_{\omega})-\overline{s}^{y}(F_{\omega})\rvert}\left(\sum_{e\in\mathbf{C}_\omega^k}\overline{\psi}_\omega^t(e)\right)\left(\sum_{g\in\mathbf{C}_{\sigma^k\omega}^{l}}\overline{\psi}_{\sigma^k\omega}^t(g)\right)
\\&=\gamma_{\max}^{(k+l)\lvert \overline{s}^{x}(F_{\omega})-\overline{s}^{y}(F_{\omega})\rvert}\overline{\Psi}_\omega^k(t)\overline{\Psi}_{\sigma^k\omega}^l(t).
\end{align*}
The other inequality follows by a similar argument, again noting that \[\gamma_{\max}^{(k+l)\lvert \overline{s}^{x}(F_{\omega})-\overline{s}^{y}(F_{\omega})\rvert}=1\] for all $\omega\in\Lambda^{\N}$ in the separated case and for all $\omega\in\mathfrak{G}$ in the non-separated case.
\end{proof}

For the proof of Lemma~\ref{lma:spressure} we require the following proposition, a variant of the \emph{subadditive ergodic theorem}, originally due to Kingman~\cite{Kingman73}.
\begin{prop}[Derriennic,~\cite{Derriennic83}]\label{KingmanProp}
Let $X_{m}(\omega)$ be a family of measurable and integrable random variables on a probability space $(\Omega,\mu)$. Assume that \[\inf_{n}(1/n)\E(X_{n})>-\infty\] and that the subadditive defects are bounded by a family of positive random variables $Y_{m}$, that is
\[
X_{n+m}(\omega)-X_{n}(\omega)-X_{m}(\sigma^{n}\omega)\leq Y_{m}(\sigma^{n}\omega)\;\; \text{a.s.},
\]
for all $n,m$ with $\sup_{m} \E Y_{m}<\infty$. Then there exists a $\mu$-invariant random variable $\overline Y$ such that $\lim_{n}(1/n)X_{n}=\overline Y$ and $\lim_{n}(1/n)\E X_{n}=\E\overline Y$.
\end{prop}

\begin{proof}[Proof (of Lemma~\ref{lma:spressure})]
Let $t\in[0,\overline{s}^{x}(F_{\omega})+\overline{s}^{y}(F_{\omega})]$, and consider $\log\overline{\Psi}_{\omega}^{k}(t)$. 
We first note that $\log\overline{\Psi}_{\omega}^{k}(t)$ is a (measurable) random variable. Since there exists $c>0$ such that $\overline{\psi}_{\omega}^{t}(x)\geq c$ for all $i$ and $x\in\mathbb{I}_{i}$, we can write
\[
\inf_{k}\frac{\E\log\overline{\Psi}_{\omega}^{k}(t)}{k}\geq\inf_{k}\frac{\log c^{k}}{k}=\log c>-\infty.
\]
Hence the condition on bounded infimum is satisfied.
Obviously the shift map $\sigma$ is an invariant and ergodic transformation on $\Omega=\Lambda^{\N}$, and the subadditive defect is
\begin{align*}
\log\overline{\Psi}_{\omega}^{k+l}(t)-\log\overline{\Psi}_{\omega}^{k}(t)-\log\overline{\Psi}_{\sigma^{k}\omega}^{l}(t)
&\leq
\log\gamma_{\max}^{(k+l)\lvert \overline{s}^{x}(F_{\omega})-\overline{s}^{y}(F_{\omega})\rvert}\\
&=0\;\;\text{(a.s.)}&&\text{(by Lemma~\ref{lma:stochAdditivity})}
\end{align*}
and therefore $\overline{P}_{\omega}(t)$ converges to some random variable $\widehat{P}_{\omega}$ almost surely. Using ergodicity of $\sigma$ we can conclude that $\widehat{P}_{\omega}$ is almost surely constant and so, almost surely,
\begin{equation}\label{infimumPressure}
\overline P_{\omega}(t)=\lim_{k}\exp\left(\frac{\E\log\overline{\Psi}_{\omega}^{k}}{k}\right).
\end{equation}
The case for $t>\overline{s}^{x}(F_{\omega})+\overline{s}^{y}(F_{\omega})$ follows by considering the stochastically subadditive sequence of $-\log\overline{\Psi}_{\omega}^{k}$ and we will omit details here. We simply comment that $\overline\Psi_{\omega}^{k}=\underline\Psi_{\omega}^{k}$ (a.s.) and so $\underline P_{\omega}(t)=\overline P_{\omega}(t)$ almost surely, and we denote this common, almost sure, constant value by $P(t)$.
\end{proof}

\begin{lma}\label{continuityLemma}
The $s$-pressure $P(s)$ is strictly decreasing and continuous in $s$, and there exists unique $s_{B}$ such that $P(s_{B})=1$.
\end{lma}
\begin{proof}
Let $\underline \alpha=\min_{e\in\mathbb{I}_{i},i\in\Lambda}\alpha_{m}(e)$ and $\overline \alpha=\max_{e\in\mathbb{I}_{i},i\in\Lambda}\alpha_{M}(e)$. We have for $\epsilon>0$,
\begin{align*}
P(s+\epsilon)&=\ess\lim_{k\to\infty}\left(\sum_{e\in\mathbf{C}_{\omega}^{k}}\alpha_{M}(e)^{\overline{s}(e,F_{\omega})}\alpha_{m}(e)^{s+\epsilon-\overline{s}(e,F_{\omega})}\right)^{1/k}\\
&\leq\ess\lim_{k\to\infty}\left(\sum_{e\in\mathbf{C}_{\omega}^{k}}\overline{\alpha}^{k\epsilon}\alpha_{M}(e)^{\overline{s}(e,F_{\omega})}\alpha_{m}(e)^{s-\overline{s}(e,F_{\omega})}\right)^{1/k}=\overline{\alpha}^{\epsilon}P(s).
\end{align*}
Similarly we can establish the lower bound to get 
\[
\underline{\alpha}^{\epsilon}P(s)\leq P(s+\epsilon)\leq\overline{\alpha}^{\epsilon}P(s).
\]
It immediately follows that $P$ is continuous and strictly decreasing. Furthermore, letting $\epsilon\to\infty $ we can see that $P(s)\to 0$ as $s\to \infty$. Finally consider 
\begin{align*}
P(0)&=\ess\lim_{k\to\infty}\left(\sum_{e\in\mathbf{C}_{\omega}^{k}}\alpha_{M}(e)^{\overline{s}(e,F_{\omega})}\alpha_{m}(e)^{-\overline{s}(e,F_{\omega})}\right)^{1/k}\\
&\geq \ess\lim_{k\to\infty}\left(\sum_{e\in\mathbf{C}_{\omega}^{k}}1\right)^{1/k}=\exp\E(\log(\#\mathbb{I}_{i}))>1.
\end{align*}
The last inequality follows by our assumption that at least one of our IFSs contains two maps. We can therefore conclude that $s_{B}$ in~(\ref{unityPressure}) is unique, which concludes the proof.
\end{proof}

To prove that $s_{B}$ is the almost sure box-counting dimension of $F_{\omega}$ we first define a useful stopping.
\begin{defn}
For $0<\delta\leq1$ we define the \emph{$\delta$-stopping}:
\[ 
\Xi_{\omega}^{\delta}=\{e=e_{1}\dots e_{\lvert e\rvert}\in\mathbf{C}_{\omega}^* \mid \alpha_m(e)\leq\delta \text{ and }\alpha_m(e_{1}e_{2}\dots e_{\lvert e\rvert-1})>\delta\},
\]
\end{defn}
This is the collection of arrangements of words such that their associated rectangle has shorter side comparable to $\delta$, i.e.\ for $e\in \Xi_{\omega}^{\delta}$,
\begin{equation}\label{alphaboundsEq}
\underline\alpha\delta<\alpha_m(e)\leq\delta.
\end{equation}

We can now prove the upper bound of (\ref{mainEq}).
\begin{lma}\label{upperBoundLemma}
Let $F_{\omega}$ be the attractor of a box-like self-affine random iterated function system. Irrespective of overlaps, almost surely,
\[
\dim_{B}F_{\omega}\leq s_{B}.
\]
\end{lma}
\begin{proof}
For $e\in\mathbf{C}_{\omega}^{*}$ define 
\[
F_{\omega}(e)=f(e,F_{\sigma^{\lvert e\rvert}\omega}).
\]
Let $\epsilon>0$ be arbitrary and let $\{U_{e,i}\}_{i=1}^{N_{\epsilon}(F_{\omega}(e))}$ be a minimal $\epsilon$-cover of $F_{\omega}(e)$; then
\begin{equation}\label{coveringEstimate}
F_{\omega}\subseteq\bigcup_{e\in\Xi_{\omega}^{\epsilon}}\bigcup_{i=1}^{N_{\epsilon}(F_{\omega}(e))} U_{e,i}.
\end{equation}
Recall that $\pi_{e}$ is the projection onto the axes parallel to the longest side of $f(e,\Delta)$. Using~(\ref{coveringEstimate}) we get
\begin{align*}
N_{\epsilon}(F_{\omega})\leq\sum_{e\in\Xi_{\omega}^{\epsilon}}N_{\epsilon}(F_{\omega}(e))
=\sum_{e\in\Xi_{\omega}^{\epsilon}}N_{\epsilon}(f(e,F_{\sigma^{\lvert e\rvert}\omega}))
=\sum_{e\in\Xi_{\omega}^{\epsilon}}N_{\epsilon}(f(e,\pi_{e}F_{\sigma^{\lvert e\rvert}\omega})),
\end{align*}
since the rectangles have shortest length equal to $\epsilon$ or less.
Now notice that for $\epsilon>0$ there exists $C_{\omega}>0$ such that
\[
C_{\omega}^{-1}r^{-\underline{s}_{\omega}(e)+\epsilon/2}\leq N_{r}(\pi_{e}F_{\omega})\leq C_{\omega}r^{-\overline{s}_{\omega}(e)-\epsilon/2},
\]
with $0< C_{\omega}<\infty$ holding almost surely.
We thus get 
\[
N_{\epsilon}(F_{\omega})\leq\sum_{e\in\Xi_{\omega}^{\epsilon}}N_{\epsilon/\alpha_{M}(e)}((\pi_{e}F_{\sigma^{\lvert e\rvert}\omega}))
\leq C_{\omega}\sum_{e\in\Xi_{\omega}^{\epsilon}}\left(\frac{\epsilon}{\alpha_{M}(e)}\right)^{- \overline{s}(e,F_{\sigma^{\lvert e\rvert}\omega})-\epsilon/2},
\]
and as $\underline \alpha \epsilon<\alpha_{m}(e)$ for all $e\in\Xi_{\omega}^{\epsilon}$ we deduce, for $C^{*}=\underline{\alpha}^{-(s_{B}+\epsilon)}$,
\begin{align*}
N_{\epsilon}(F_{\omega})&\leq C_{\omega}\left(\frac{\alpha_{m}(e)}{\underline{\alpha}\epsilon}\right)^{s_{B}+\epsilon}\sum_{e\in\Xi_{\omega}^{\epsilon}}\left(\frac{\epsilon}{\alpha_{M}(e)}\right)^{- \overline{s}(e,F_{\sigma^{\lvert e\rvert}\omega})-\epsilon/2}\\
&\leq C_{\omega}\underline{\alpha}^{-(s_{B}+\epsilon)}\epsilon^{-(s_{B}+\epsilon)}\sum_{e\in\Xi_{\omega}^{\epsilon}}\alpha_{M}(e)^{\overline{s}(e,F_{\omega})+\epsilon/2}\alpha_{m}(e)^{s_{B}+\epsilon}\epsilon^{-\overline{s}(e,F_{\omega})-\epsilon/2}\\
&\leq C_{\omega}C^{*}\epsilon^{-(s_{B}+\epsilon)}\sum_{e\in\Xi_{\omega}^{\epsilon}}\alpha_{M}(e)^{\overline{s}(e,F_{\omega})+\epsilon/2}\alpha_{m}(e)^{s_{B}+\epsilon-\overline{s}(e,F_{\omega})-\epsilon/2}\\
&\leq C_{\omega}C^{*}\epsilon^{-(s_{B}+\epsilon)}\sum_{e\in\Xi_{\omega}^{\epsilon}}\overline\psi_{\omega}^{s_{B}+\epsilon/2}(e)\\
&\leq C_{\omega}C^{*}\epsilon^{-(s_{B}+\epsilon)}\sum_{k=1}^{\infty}\overline\Psi_{\omega}^{k}(s_{B}+\epsilon/2).
\end{align*}
But since $ P(s+\epsilon)<1$ there exists $C_{\omega}'$ for almost every $\omega$ such that 
\[
\overline\Psi_{\omega}^{k}(s_{B}+\epsilon)\leq C_{\omega}' {P}(s+\epsilon/2)^{k}.
\]
Therefore, almost surely,
\[
N_{\epsilon}(F_{\omega})\leq C^{*}C_{\omega}C_{\omega}'\epsilon^{-(s_{B}+\epsilon)}\sum_{k}P(s+\epsilon/2)^{k}<\infty
\]
and hence $\dim_{B}(N_{\epsilon}(F_{\omega}))\leq s_{B}+\epsilon$ as required.
\end{proof}

For $t<s_{B}$ the sum over the random modified singular function of the elements in the stopping is bounded from below.

\begin{lma}\label{complicatedLowerLemma}
Let 
\[
L^{\delta}_{\omega}(t)=\sum_{e\in\Xi_{\omega}^{\delta}}\overline\psi_\omega^t(e).
\]
Then $P(t)<1$ implies $L^{\delta}_{\omega}(t)<1$ for small enough $\delta$, and $P(t)>1$ implies $L^{\delta}_{\omega}(t)>1$ for small enough $\delta$ almost surely.
\end{lma}

\begin{proof}
We start by introducing the same notation as in~\cite{Troscheit15} and write the stopping $\Xi_{\omega}^{\delta}$ in an infinite matrix fashion. Define the matrix $\widehat{\Xi}_{\omega}^{\delta}$ entrywise for $i,j\in\N$ by 
\[
(\widehat\Xi_{\omega}^{\delta})_{i,j}=\bigsqcup\left\{e\in\Xi_{\sigma^{i-1}\omega}^{\delta}\,:\,\lvert e\rvert=\max\{0,j-i\}\right\}
\]
if there exists $e\in\Xi_{\sigma^{i-1}\omega}^{\delta}$ with $\lvert e\rvert=\max\{0,j-i\}$. Otherwise set $(\widehat\Xi_{\omega}^{\delta})_{i,j}=\varnothing$.
We define the vector $\mathds{1}_{\epsilon_{0}}^{k}$ by
\[
(\mathds{1}_{\epsilon_{0}}^{k})_{i}=
\begin{cases}
\epsilon_{0},&\text{ if }i=k,\\
\varnothing,&\text{ otherwise.}
\end{cases}
\]

Given $0<\xi<1$, for every $0<\delta<1$ there exists a unique $k\in\N_{0}$ such that $\delta=\xi^{k}\theta$, for $\xi<\theta\leq1$. Fix such a $\xi$, we start by showing that 
\begin{equation}\label{theLAsymp}
L_{\omega}^{\delta}(t)=\sum_{e\in(\mathds{1}^{1}_{\epsilon_{0}}\widehat\Xi_{\omega}^{\delta})}\overline{\psi}_{\omega}^{t}(e)
\;\;\;\asymp \sum_{e\in(\mathds{1}^{1}_{\epsilon_{0}}\widehat\Xi_{\omega}^{(\xi^{k})})}\overline{\psi}_{\omega}^{t}(e),
\end{equation}
where we write $g\asymp h$ to indicate $g/h$ and $h/g$ are bounded uniformly away from $0$ in $\delta$ (and thus $k$). The first equality in~(\ref{theLAsymp}) is immediate as 
\[
\bigsqcup_{i}(\mathds{1}^{1}_{\epsilon_{0}}\widehat\Xi_{\omega}^{\delta})_{i}=\Xi_{\omega}^{\delta}.
\] 
For the asymptotic equality note that there exists $c_{1}>0$ so that for all $\delta>0$ and all $e\in\Xi_{\omega}^\delta$ and $g\in\beth^{*}$ such that $e\odot g\in\Xi_{\omega}^\delta$ with $\lvert g\rvert=1$ we have
\begin{equation}\label{boundedSingular}
c_{1}^{-1} \overline{\psi}^{t}_{\omega}(e)\leq\overline{\psi}^{t}_{\omega}(e\odot g)\leq c_{1}\overline{\psi}^{t}_{\omega}(e).
\end{equation}
Now also note that for every $e\in\Xi_{\omega}^{\delta}$ there exists unique $e^{\dagger}\in\Xi_{\omega}^{\xi^{k}}$ and $e^{\ddagger}\in\Xi_{\omega}^{\xi^{k+1}}$ such that the cylinders satisfy
\begin{equation}\label{cylinderInclusions}
[e^{\ddagger}]\subseteq[e]\subseteq[e^{\dagger}].
\end{equation}
 There also exists integer $n_{\max}(\xi)$, independent of $\omega\in\Lambda^{\N}$, such that for all $g\in\Xi_{\omega}^{\xi} $ we have $\lvert g\rvert<n_{\max}(\xi)$ as all maps are contractions. We find the following bounds, where, for $e\in\Xi_{\omega}^{\xi^{k}}$, the $g$ is such that $e\odot g\in\Xi_{\omega}^{\delta}$,
\begin{align*}
\sum_{e\in(\mathds{1}^{1}_{\epsilon_{0}}\widehat\Xi_{\omega}^{\delta})}\hspace{-0.2cm}\overline{\psi}_{\omega}^{t}(e)
&=\sum_{e\in(\mathds{1}^{1}_{\epsilon_{0}}\widehat\Xi_{\omega}^{\xi^{k}})}\sum_{\{g \,\mid\, e\odot g\in\Xi_{\omega}^{\delta}\}}
\hspace{-0.3cm}\overline{\psi}_{\omega}^{t}(e\odot g)\\
&\leq \;\;c_{1}^{n_{\max}(\xi)}\hspace{-0.2cm}\sum_{e\in(\mathds{1}^{1}_{\epsilon_{0}}\widehat\Xi_{\omega}^{\xi^{k}})}\sum_{\{g \,\mid\, e\odot g\in\Xi_{\omega}^{\delta}\}}
\hspace{-0.2cm}\overline{\psi}_{\omega}^{t}(e)
\\&\leq\;\;
(n_{\max}(\xi)c_{1})^{n_{\max}(\xi)}\hspace{-0.2cm}
\sum_{e\in(\mathds{1}^{1}_{\epsilon_{0}}\widehat\Xi_{\omega}^{\xi^{k}})}\overline{\psi}_{\omega}^{t}(e),
\end{align*}
and
\begin{align*}
\sum_{e\in(\mathds{1}^{1}_{\epsilon_{0}}\widehat\Xi_{\omega}^{\delta})}\hspace{-0.2cm}\overline{\psi}_{\omega}^{t}(e)
&=\sum_{e\in(\mathds{1}^{1}_{\epsilon_{0}}\widehat\Xi_{\omega}^{\xi^{k}})}\sum_{\{g \,\mid\, e\odot g\in\Xi_{\omega}^{\delta}\}}
\hspace{-0.3cm}\overline{\psi}_{\omega}^{t}(e\odot g)\\
&\geq c_{1}^{-n_{\max}(\xi)}\hspace{-0.2cm}
\sum_{e\in(\mathds{1}^{1}_{\epsilon_{0}}\widehat\Xi_{\omega}^{(\xi^{k})})}\overline{\psi}_{\omega}^{t}(e)
.
\end{align*}
Hence the asymptotic estimate in~(\ref{theLAsymp}) holds.

Now, by (\ref{boundedSingular}) and (\ref{cylinderInclusions}) we have for some $c_{2}>0$ that $L_{\omega}^{\delta}(t)$ is related to the sum of the modified singular value function over the $\xi$ approximation codings by
\begin{equation}\label{approxCodingBounds}
c_{2}^{-k}\mathfrak{L}^{\delta,k}_{\omega}(t)\leq L_{\omega}^{\delta}(t)\leq c_{2}^{k}\mathfrak{L}^{\delta,k}_{\omega}(t),
\end{equation}
where
\[
\mathfrak{L}^{\delta,k}_{\omega}(t)=\sum_{e\in\tensor*[_{k}]{\widehat\Xi}{_{\omega}^{\xi}}}\overline{\psi}_{\omega}^{t}(e)
\text{ and }\;
\tensor*[_{k}]{\widehat\Xi}{_{\omega}^{\xi}}=\mathds{1}^{k}_{\epsilon_{0}}\underbrace{\widehat\Xi_{\omega}^{\xi}\dots\widehat\Xi_{\omega}^{\xi}}_{k\text{ times}}.
\]
Now consider the infinite matrix $\mathcal{M}_{\omega}^{\xi}(t)$ that is defined for all $i,j\in\N$ by
\[
(\mathcal{M}_{\omega}^{\xi}(t))_{i,j}=\sum_{e\in(\tensor*[_{i}]{\widehat\Xi}{_{\omega}^{\xi}})_{j}}\overline{\psi}_{\omega}^{t}(e).
\]
By definition the sum over all entries in the $k$-th row is $\mathfrak{L}_{\omega}^{\delta,k}$ and we now show that the sums of the $k$-th column is related to $\overline{\Psi}_{\omega}^{k}(t)$, in fact it is easy to see that every entry of the $k$-th column of $\mathcal{M}_{\omega}(t)^{\xi}$ is a lower bound to $\overline{\Psi}_{\omega}^{i}(t)$ as every such entry is given by a word of length $k$. 
 The number of non-empty column entries is at most $n_{\max}(\xi)k$, where $k$ is the column index. Combining this with (\ref{boundedSingular}) and (\ref{cylinderInclusions}) we get for some $c_{3}>0$
\begin{equation}\label{psiupperbound}
\overline{\Psi}_{\omega}^{k}(t)=\sum_{e\in\mathbf{C}_{\omega}^{k}}\overline{\psi}_{\omega}^{t}(e)\geq \frac{c_{3}^{-1}c_{1}^{-n_{\max}(\xi)}}{n_{\max}(\xi)}\sum_{j=0}^{n_{\max}(\xi)-1}\sum_{i=1}^{\infty}(\mathcal{M}_{\omega}^{\xi}(t))_{i,j+k} .
\end{equation}
Similarly for every $e\in\mathbf{C}^{k}_{\omega}$ there exists $g\in(\tensor*[_{k+j}]{\widehat\Xi}{_{\omega}^{\xi}})$ for some $j\in\{0,\dots,n_{\max}(\xi)-1\}$ such that $[e]\subseteq[g]$ and using (\ref{cylinderInclusions}) and (\ref{boundedSingular}) again we get for some $c_{4}>0$ that
\[
\overline{\Psi}_{\omega}^{k}(t)=\sum_{e\in\mathbf{C}_{\omega}^{k}}\overline{\psi}_{\omega}^{t}(e)\leq c_{4}c_{1}^{n_{\max}(\xi)}\sum_{j=0}^{n_{\max}(\xi)-1}\sum_{i=1}^{\infty}(\mathcal{M}_{\omega}^{\xi}(t))_{i,j+k}.
\]
We call $\sum_{j=0}^{n_{\max}(\xi)-1}\sum_{i=1}^{\infty}(\mathcal{M}_{\omega}^{\xi}(t))_{i,j+k}$ the \emph{$n_{\max}(\xi)$-corridor at $k$} and denote by $\mathfrak{C}_{\xi}(l)$ all pairs $(i,j)\in\N^{2}$ such that $l\leq i< l+n_{\max}(\xi)$, that is the pairs in $\mathfrak{C}_{\xi}(l)$ are the coordinates of the $n_{\max}(\xi)$-corridor at $l$.

The final ingredient is to compare the rate of growth of the sum of the singular value function of non-empty column entries with the rate of growth of the sum over non-empty row entries. First notice that the rate of growth of the rows is related to the maximal element in the row by 
\[
\max_{j}(\mathcal{M}_{\omega}^{\xi}(t))_{i,j}\leq
\sum_{j}(\mathcal{M}_{\omega}^{\xi}(t))_{i,j}\leq
n_{\max}(\xi)i \max_{j}(\mathcal{M}_{\omega}^{\xi}(t))_{i,j}.
\]
Note that elements in the matrix cannot increase arbitrarily from row to row, that is we have for some $c_{5}>0$ and all integers $i,j>1$,
\begin{equation}\label{arbitraryIncrease}
(\mathcal{M}_{\omega}^{\xi}(t))_{i,j}\leq c_{5} \max_{k\in\{1,\dots, n_{\max}(\xi)\}}(\mathcal{M}_{\omega}^{\xi}(t))_{j-1,i-k}.
\end{equation}
Combining this with the fact that at least one of $(\mathcal{M}_{\omega}^{\xi}(t))_{j-1,i-k}$ is positive, the maximal element cannot move arbitrarily and for every column $k$ there exists a row $r$ such that
\[
\max_{i\in\{0,\dots, n_{\max}(\xi)\}}(\mathcal{M}_{\omega}^{\xi}(t))_{r,k+i}=\max_{j\in\{0,1,\dots\}}(\mathcal{M}_{\omega}^{\xi}(t))_{r,j}.
\]
But using the existence of $n_{\max}(\xi)$ and that the $\Xi_{\omega}^{\xi}$ do not contain the empty word, we also have 
\(
(\mathcal{M}_{\omega}^{\xi}(t))_{i,j}=0
\)
for $j<i$ and $j>n_{\max}(\xi)i$. We deduce that
\begin{align}
L_{\omega}^{\delta}(t)&\leq c_{2}^{k}\mathfrak{L}_{\omega}^{\delta,k}(t)=c_{2}^{k}\sum_{e\in\tensor*[_{k}]{\widehat\Xi}{_{\omega}^{\xi}}}\overline{\psi}_{\omega}^{t}(e)&&\text{by (\ref{approxCodingBounds})}\nonumber\\
&=c_{2}^{k}\sum_{j}(\mathcal{M}_{\omega}^{\xi}(t))_{k,j}\nonumber\\
&\leq c_{2}^{k}n_{\max}(\xi)k(\mathcal{M}_{\omega}^{\xi}(t))_{k,j_{\max}}&&\text{ for some }j_{\max}\in\{k,\dots,n_{\max}(\xi)k\}\nonumber\\
&\leq c_{2}^{k}n_{\max}(\xi)k\sum_{l=0}^{n_{\max}(\xi)-1}\sum_{i=1}^{\infty}(\mathcal{M}_{\omega}^{\xi}(t))_{i,l+j_{\max}}\nonumber\\
&\leq c_{2}^{k}c_{3}c_{1}^{n_{\max}(\xi)}n_{\max}(\xi)^{2}k\overline{\Psi}_{\omega}^{j_{\max}}(t)&&\text{by (\ref{psiupperbound})}.\nonumber
\end{align}
Similarly, we can derive the lower bound,
\begin{align}
L_{\omega}^{\delta}(t)&\geq c_{2}^{-k}\mathfrak{L}_{\omega}^{\delta,k}(t)=c_{2}^{-k}\sum_{e\in\tensor*[_{k}]{\widehat\Xi}{_{\omega}^{\xi}}}\overline{\psi}_{\omega}^{t}(e)&&\text{by (\ref{approxCodingBounds})}\nonumber\\
&=c_{2}^{-k}\sum_{j}(\mathcal{M}_{\omega}^{\xi}(t))_{k,j}\nonumber\\
&\geq c_{2}^{-k}(\mathcal{M}_{\omega}^{\xi}(t))_{k,j_{\max}}&&\hspace{-4cm}\text{ for maximising }j_{\max}\in\{k,\dots,n_{\max}(\xi)k\}\nonumber\\
&\geq c_{2}^{-k}c_{5}^{-k-j_{\max}-n_{\max}(\xi)}\max_{(i,j)\in\mathfrak{C}_{\xi}(j_{\max})}(\mathcal{M}_{\omega}^{\xi}(t))_{i,j}\label{complicatedDerivation}\\
&\geq \frac{c_{2}^{-k}c_{5}^{-k-n_{\max}(\xi)k-n_{\max}(\xi)}}{n_{\max}(\xi)(j_{\max}+n_{\max}(\xi))}\sum_{j=0}^{n_{\max}(\xi)-1}\sum_{i=1}^{\infty}(\mathcal{M}_{\omega}^{\xi}(t))_{i,j+j_{\max}}\nonumber\\
&\geq \frac{c_{2}^{-k}c_{5}^{-k-n_{\max}(\xi)k-n_{\max}(\xi)}}{ c_{4}c_{1}^{n_{\max}(\xi)}n_{\max}(\xi)(j_{\max}+n_{\max}(\xi))}\overline{\Psi}_{\omega}^{j_{\max}}(t)\nonumber.
\end{align}
The inequality in (\ref{complicatedDerivation}) arises as the maximal element in the $n_{\max}(\xi)$-corridor at $j_{\max}$ must be in one of $j_{\max}+n_{\max}(\xi)$ rows and the maximal element can be no larger than $c_{5}^{j_{\max}+n_{\max}(\xi)}$ times the maximal element in any of the preceding rows by~(\ref{arbitraryIncrease}). 

Thus we get upper and lower bounds to $L_{\omega}^{\delta}(t)$ and we find for some $c_{6},c_{7}>0$ such that
\[
\frac{\log L_{\omega}^{\delta}(t)}{-\log\delta}\leq \frac{\log c_{6}c_{7}^{k}\overline{\Psi}_{\omega}^{j_{\max}}(t)}{-\log\xi^{k}\theta}
\leq\frac{(1/k)\log c_{6}+\log c_{7}+(j_{\max}/k)\log(\overline{\Psi}_{\omega}^{j_{\max}}(t)^{1/j_{\max}})}{-\log\xi}.
\]
Thus, for arbitrary $\epsilon>0$, we can pick $\xi$ small enough such that
\[
\frac{\log L_{\omega}^{\delta}(t)}{-\log\delta}\leq \frac{(j_{\max}/k)\log(\overline{\Psi}_{\omega}^{j_{\max}}(t)^{1/j_{\max}})}{-\log\xi}+\frac{\epsilon}{2}.
\]
Similarly, for some $c_{8},c_{9}>0$ and small enough $\xi$,
\[
\frac{\log L_{\omega}^{\delta}(t)}{-\log\delta}\geq \frac{\log c_{8}c_{9}^{k}\overline{\Psi}_{\omega}^{j_{\max}}(t)}{-k\log\xi}\geq\frac{(j_{\max}/k)\log(\overline{\Psi}_{\omega}^{j_{\max}}(t)^{1/j_{\max}})}{-\log\xi}-\frac{\epsilon}{2}
\]
Observe now that almost surely 
\[
\limsup_{k} (j_{\max}/k)\log(\overline{\Psi}_{\omega}^{j_{\max}}(t)^{1/j_{\max}})\leq \max_{p\in\{+1,-1\}}n_{\max}(\xi)^{p}\log P(t).
\]
Similarly, almost surely,
\[
\liminf_{k} (j_{\max}/k)\log(\overline{\Psi}_{\omega}^{j_{\max}}(t)^{1/j_{\max}})\geq \min_{p\in\{+1,-1\}}n_{\max}(\xi)^{p}\log P(t).
\]
Now as $\epsilon$ was arbitrary we conclude, for small enough $\delta>0$, that 
\[(\log L_{\omega}^{\delta}(t))/(-\log\delta)>0\text{ if }P(t)>1,\] 
and 
\[(\log L_{\omega}^{\delta}(t))/(-\log\delta)<0\text{ if }P(t)<1.
\] 
Therefore the implications in the statement hold.
\end{proof}

\begin{lma}\label{lma:mainLower}
Let $(\mathbb{I},\mathbf{p})$ be a box-like self-affine RIFS that satisfies the uniform open rectangle condition. Let $F_{\omega}$ be the associated $1$-variable random set. Then
\[
\dim_{B}F_{\omega}\geq s_{B}\;\;\;\text{(a.s.)}.
\]
\end{lma}

\begin{proof}
Let $\delta>0$ and consider the $\delta$-mesh on $\Delta$, denoted by $\mathcal{N}_{\delta}$. Since we assume the uniform open rectangle condition, the open rectangles $\mathcal{C}=\{f_{e}(\mathring\Delta)\}_{e\in\Xi_{\omega}^{\delta}}$ are pairwise disjoint. 
Furthermore the side lengths of the rectangles $R\in\mathcal C$ are bounded below by $\underline \alpha\delta$ by definition and thus the number of rectangles of $\mathcal{C}$ each square in the grid of $\mathcal{N}_{\delta}$ can intersect is at most $C^{-1}=(\underline\alpha^{-1}+2)^{2}$. Thus 
\[
M_{\delta}(F_{\omega})\geq C\sum_{e\in\Xi_{\omega}^{\delta}}N_{\delta}(F_{\omega}(e)).
\]
In a similar fashion to the upper bound proof, Lemma~\ref{upperBoundLemma}, we find
\[
M_{\delta}(F_{\omega})\geq C\sum_{e\in\Xi_{\omega}^{\delta}}N_{\delta/\alpha_{M}(e)}(F_{\sigma^{\lvert e\rvert}\omega})
\geq C C_{\omega}^{-1}\sum_{e\in\Xi_{\omega}^{\epsilon}}\left(\frac{\delta}{\alpha_{M}(e)}\right)^{-\underline s(e,F_{\omega})+\epsilon/2}.
\]
Using~(\ref{alphaboundsEq}),
\begin{align*}
M_{\delta}(F_{\omega})&\geq C C_{\omega}^{-1}\sum_{e\in\Xi_{\omega}^{\epsilon}}\left(\frac{\underline\alpha^{-1} \alpha_{m}(e)}{\alpha_{M}(e)}\right)^{-\underline s(e,F_{\omega})+\epsilon/2}\\
&\geq C C_{\omega}^{-1}\left(\frac{\alpha_{m}(e)}{\delta}\right)^{s_{B}-\epsilon}\sum_{e\in\Xi_{\omega}^{\epsilon}}\left(\frac{\underline\alpha^{-1} \alpha_{m}(e)}{\alpha_{M}(e)}\right)^{-\underline s(e,F_{\omega})+\epsilon/2}.
\end{align*}
Thus, for some $C^{*}_{\omega}>0$,
\begin{align*}
M_{\delta}(F_{\omega})&\geq C C_{\omega}^{-1}\delta^{-(s_{B}-\epsilon)}\underline\alpha^{\max_{z\in\{x,y\}}\underline s^{z}(F_{\omega})-\epsilon/2}\sum_{e\in\Xi_{\omega}^{\delta}}\alpha_{M}(e)^{\underline s(e,F_{\omega})}\alpha_{m}(e)^{s_{B}-\epsilon/2-\underline s(e,F_{\omega})}\\
&\geq C^{*}_{\omega}C_{\omega}^{-1}\delta^{-(s_{B}-\epsilon/2)}\sum_{e\in\Xi_{\omega}^{\delta}}\underline\psi_{\omega}^{s_{B}-\epsilon/2}
=C^{*}_\omega C_{\omega}^{-1}\delta^{-(s_{B}-\epsilon/2)}L^{\delta}_{\omega}(s_{B}-\epsilon/2).
\end{align*}
This in turn gives
\[
\frac{\log M_{\delta}(F_{\omega})}{-\log\delta}\geq\frac{\log C^{*}_\omega C_{\omega}}{-\log\delta}+(s_{B}-\epsilon/2)+\frac{\log L^{\delta}_{\omega}(s_{B}-\epsilon/2)}{-\log\delta}
\]
The lower bound follows almost surely because the first term becomes arbitrarily small, and by Lemma~\ref{complicatedLowerLemma}, the last term is positive for small enough $\delta$.
\end{proof}

\subsection{Proofs for Section~\ref{inftyVarRes}}

In this Section we prove the remaining results concerning $\infty$-variable box-like self-affine carpets.
We define a random variable $Y_{k}^{t}$ and show that it behaves similarly to a martingale. For $t>s_{B}$ we have $Y_{k}^{t}\to 0$ a.s.\ and $t$ is an almost sure lower bound for the box-counting dimension of $F_{\tau}$.
We define a new random variable $Z_{k}^{t}$ that, for $t<s_{B}$, increases exponentially a.s.. This will then allow us to establish  the lower bound.

\begin{lma}\label{ItsAMartingale}
Let $Y_{k}^{t}$ be the random variable given by
\[
\hspace{2cm}Y_{k}^{t}(\tau)=\sum_{\xi\in\mathbf{T}_{\tau}^{k}}\overline{\psi}_{\tau}^{t}(\xi).\hspace{2cm}(\tau\in\mathcal{T})
\]
For all $t\in\R^{+}_{0}$ and $l,k\in\N$ and some random variable $c_{k,l}$ that will be defined in~(\ref{approximateConstant}), the sequence of random variables satisfies
\begin{equation}\label{infinityQuasiMartingale}
\E\left(Y_{k+l}^{t}(\tau)\given[\Large] \mathcal{F}^{k}\right)=c_{k,l}(\tau) Y_{k}^{t}(\tau) \E(Y_{l}^{t}).
\end{equation}
Here $\mathcal{F}^{k}$ refers to the filtration corresponding to the `knowledge of outcomes up to the $k$-th level'.
Furthermore $0<c_{k,l}\leq 1$ for $t\in[0,s^{x}+s^{y}]$ and so the sequence $\{Y_{ql}^{t}\}_{q=1}^{\infty}$ forms a supermartingale if $\E(Y_{l}^{t})\leq1$.
\end{lma}
\begin{proof}
Let $e\in\mathbf{T}_{\tau}^{k}$ and $g\in\mathbf{T}_{\sigma^{e}\tau}^{l}$, then there exists  $C(e,g)$ such that
\[
\overline{\psi}_{\tau}^{t}(e\odot g)=C(e,g)\overline{\psi}_{\tau}^{t}(e)\overline{\psi}_{\sigma^{e}\tau}^{t}(g).
\]
We define the dual of the modified singular value function 
\[
\overline{\psi}_{\tau}^{*t}(e)=\alpha_{m}(e)^{\overline{s}(\neg e,F_{\tau})}\alpha_{M}(e)^{t-\overline{s}(\neg e,F_{\tau})},
\]where $s(\neg e,F_{\tau})$ is the box-counting dimension of the projection of $F_{\tau}$ onto the \emph{shorter} side of $f(e,\Delta)$. We find that $C(e,g)$ takes only a few possible values, depending on $e$ and $g$.
\begin{enumerate}
\item $\alpha_{M}(e\odot g)=\alpha_{M}(e)\alpha_{M}(g)$ and
\begin{enumerate}
\item the RIFS is of non-separated type.

\item the RIFS is of separated type and $\pi_{e}=\pi_{g}$.

\end{enumerate}
\item $\alpha_{M}(e\odot g)=\alpha_{M}(e)\alpha_{m}(g)$.

\item $\alpha_{M}(e\odot g)=\alpha_{m}(e)\alpha_{M}(g)$.
\end{enumerate}
Case (1a). Since the RIFS is of non-separated type, for almost all $\tau\in\mathcal{T}$, we have $s^{x}=s^{y}$. Furthermore, almost surely, 
\begin{align*}
\overline{\psi}_{\tau}^{t}(e\odot g)&=\alpha_{M}(e\odot g)^{s^{x}}\alpha_{m}(e\odot g)^{t-s^{x}}\\
&=(\alpha_{M}(e)\alpha_{M}(g))^{s^{x}}(\alpha_{m}(e)\alpha_{m}(g))^{t-s^{x}}\\
&=\overline{\psi}_{\tau}^{t}(e)\overline{\psi}_{e^\tau}^{t}(g),
\end{align*}
and so $C(e,g)=1$ a.s..\\
\\
Case (1b). As the RIFS is of separated type and the directions of the maximal modified singular value coincide we can apply ~(\ref{separatedRatioDecomposes}) and get $C(e,g)=1$, for all $\tau\in\mathcal{T}$.
\\\\
Case (2). We can write, for almost every $\tau\in\mathcal{T}$,
\begin{align*}
\overline{\psi}_{\tau}^{t}(e\odot g)&=\alpha_{M}(e\odot g)^{s(e\odot g)}\alpha_{m}(e\odot g)^{t-s(e\odot g)}\\
&=(\alpha_{M}(e)\alpha_{m}(g))^{s(e)}(\alpha_{m}(e)\alpha_{M}(g))^{t-s(e)}\\
&=\overline{\psi}_{\tau}^{t}(e)\alpha_{m}(g)^{s(\neg g)}\alpha_{M}(g)^{t-s(\neg g)}\\
&=\overline{\psi}_{\tau}^{t}(e)\overline{\psi}_{\sigma^{e}\tau}^{*t}(g),
\end{align*}
and so $C(e,g)=\overline{\psi}_{\sigma^{e}\tau}^{t}(g)/\overline{\psi}_{\sigma^{e}\tau}^{*t}(g)=(\alpha_{M}(g)/\alpha_{m}(g))^{t-s^{x}-s^{y}}$.
\\\\
Case (3). Similarly we can write, for almost every $\tau\in\mathcal{T}$,
\begin{align*}
\overline{\psi}_{\tau}^{t}(e\odot g)&=\alpha_{M}(e\odot g)^{s(e\odot g)}\alpha_{m}(e\odot g)^{t-s(e\odot g)}\\
&=(\alpha_{m}(e)\alpha_{M}(g))^{s(g)}(\alpha_{M}(e)\alpha_{m}(g))^{t-s(g)}
\\
&=\overline{\psi}_{\tau}^{t}(g)\alpha_{m}(e)^{s(\neg e)}\alpha_{M}(e)^{t-s(\neg e)}\\
&=\overline{\psi}_{\tau}^{*t}(e)\overline{\psi}_{\sigma^{e}\tau}^{t}(g),
\end{align*}
and so $C(e,g)=\overline{\psi}_{\tau}^{t}(e)/\overline{\psi}_{\tau}^{*t}(e)=(\alpha_{M}(e)/\alpha_{m}(e))^{t-s^{x}-s^{y}}$.

Therefore
\begin{align}
\E\left(Y_{k+l}^{t}(\tau)\given[\Large] \mathcal{F}^{k}\right)
&=\E\left(\sum_{e\in\mathbf{T}_{\tau}^{k}}\sum_{g\in\mathbf{T}_{\sigma^{e}\tau}^{l}}\overline{\psi}_{\tau}^{t}(e\odot g) \given[\Big] \mathcal{F}^{k}\right)\nonumber\\
&=\E\left(\sum_{e\in\mathbf{T}_{\tau}^{k}}\sum_{g\in\mathbf{T}_{e^\tau}^{l}}C(e,g)\overline{\psi}_{\tau}^{t}(e)\overline{\psi}_{e^\tau}^{t}(g) \given[\Big] \mathcal{F}^{k}\right)\label{MartingaleExpansionLine}\\
&=c_{k,l}(\tau)Y_{k}^{t}(\tau)\E(Y_{l}^{t}),\label{lastMartingaleLine}
\end{align}
where 
\begin{equation}\label{approximateConstant}
c_{k,l}(\tau)=\frac{\E\left(\sum_{e\in\mathbf{T}_{\tau}^{k}}\sum_{g\in\mathbf{T}_{e^\tau}^{l}}C(e,g)\overline{\psi}_{\tau}^{t}(e)\overline{\psi}_{e^\tau}^{t}(g)\given[\Big] \mathcal{F}^{k}\right)}{Y_{k}^{t}(\tau)\E(Y_{l}^{t})}.
\end{equation}
We will analyse $c_{k,l}(\tau)$ in Lemma~\ref{martingaleFactorConvergesTo1} and here only comment that by inspection we deduce $C(e,g)\leq 1$ for $t\in[0,s^{x}+s^{y}]$ and so (\ref{lastMartingaleLine}) becomes
\(
\E\left(Y_{k+l}^{t}(\tau)\given[\Large] \mathcal{F}^{k}\right)\leq Y_{k}^{t}(\tau)\E(Y_{l}^{t})
\)
and the sequence of random variables $\{Y_{ql}^{t}\}_{q=1}^{\infty}$ forms a supermartingale if $\E(Y_{l}^{t})\leq 1$.
\end{proof}


\begin{lma}\label{upperBoundProof}
Let $s\in[0,s^{x}+s^{y}]$ and $s>s_{B}$. Then the sequence $\{Y_{k}^{s}\}$ converges to $0$ exponentially fast a.s.\ and $s_{B}$ is an almost sure upper bound for the box-counting dimension of $F_{\tau}$, so
\[
\dim_{B}F_{\tau}\leq s_{B}.
\]
\end{lma}
\begin{proof}
If $s>s_{B}$ there exists $l$ such that $\E(Y_{l}^{s})<1$ by definition, see (\ref{subadditiveExpecationEqui}). Therefore, by Lemma~\ref{ItsAMartingale}, $\{Y_{ql}^{s}\}_{q=1}^{\infty}$ is a strict supermartingale. Hence $\E(Y_{ql}^{s})\to0$ as $q\to\infty$ and since $Y_{k}^{s}\asymp Y_{\lfloor k/q\rfloor q}^{s}$, we get $Y_{k}^{s}\to0$ as $k\to\infty$, almost surely. This happens at an exponential rate. That is there exists $\gamma<1$ and $D_{\tau}>0$ such that $Y_{k}\leq D_{\tau}\gamma^{k}$.

We now define a stopping set $\Xi_{\tau}^{\delta}$ analogously to before 
\[
\Xi_{\tau}^{\delta}=\{e\in\mathbf{T}_{\tau}^i \mid i\in \N, \alpha_m(e)\leq\delta \text{ and }\alpha_m(e_{1}e_{2}\dots e_{\lvert e\rvert-1})>\delta\},
\]
and we can modify the argument in Lemma~\ref{upperBoundLemma} accordingly to get, for $s=s_{B}+\delta$,
\[
N_{\epsilon}(F_{\tau})\leq C_{\tau}C^{*}\epsilon^{-(s_{B}+\delta)}\sum_{k=1}^{\infty}\sum_{e\in\mathbf{T}_{\tau}^{k}}\overline\alpha^{k\delta/2}\overline{\psi}^{s_{B}}_{\tau}(e)\leq C_{\tau}C^{*}\epsilon^{-(s_{B}+\delta)}D_{\tau}\sum_{k=1}^{\infty}\overline\alpha^{k\delta/2}\gamma^{k}<\infty,
\]
almost surely. We conclude that $s_{B}$ is an almost sure upper bound to the box-counting dimension of $F_{\tau}$.
\end{proof}

\begin{lma}\label{martingaleFactorConvergesTo1}
For $c_{k,l}(\tau)$ as in (\ref{infinityQuasiMartingale}) we find $s\in[0,s^{x}+s^{y}]$ implies $c_{k,l}\nearrow 1$ as $k\to\infty$ for almost all $\tau$. If however $s>s^{x}+s^{y}$ we get $c_{k,l}\searrow 1$ as $k\to\infty$ almost surely.
\end{lma}
\begin{proof}
We first decompose (\ref{MartingaleExpansionLine}) into
\begin{multline}\label{eq:theMonster}
\E\left( Y_{k+l}^{t}(\tau)\given \mathcal{F}^{k}\right)
=\E\left(\sum_{\substack{e\in\mathbf{T}_{\tau}^{k},g\in\mathbf{T}_{\sigma^{e}\tau}^{l}\\
\alpha_{M}(e\odot g)=\alpha_{M}(e)\alpha_{M}(g)}}\hspace{-30pt}\psi_{\tau}^{t}(e)\overline{\psi}_{\sigma^{e}\tau}^{t}(g)
\hspace{7pt}\right.\\
\left.+\hspace{-17pt}\sum_{\substack{e\in\mathbf{T}_{\tau}^{k},g\in\mathbf{T}_{\sigma^{e}\tau}^{l}\\
\alpha_{M}(e\odot g)=\alpha_{M}(e)\alpha_{m}(g)}}\hspace{-30pt}\overline{\psi}_{\tau}^{t}(e)\overline{\psi}_{\sigma^{e}\tau}^{*t}(g)
\hspace{7pt}+\hspace{-17pt}\sum_{\substack{e\in\mathbf{T}_{\tau}^{k},g\in\mathbf{T}_{\sigma^{e}\tau}^{l}\\
\alpha_{M}(e\odot g)=\alpha_{m}(e)\alpha_{M}(g)}}\hspace{-30pt}\overline{\psi}_{\tau}^{*t}(e)\overline{\psi}_{\sigma^{e}\tau}^{t}(g)\given[\Big]\mathcal{F}^{k}\right).
\end{multline}
Without loss of generality we can assume that the RIFS is strictly self-affine, that is there exists at least one $f_{i}^{j}$ such that $\alpha_{M}(e_{i}^{j})>\alpha_{m}(e_{i}^{j})$, since otherwise we trivially have $C(e,g)=1$ and so $c_{k,l}(\tau)=1$. We recall that $l$ is fixed and thus there exists a maximal ratio $\max_{\kappa\in\mathcal{T}}\max_{g\in\mathbf{T}_{\kappa}^{l}}\alpha_{M}(g)/\alpha_{m}(g)$. Now consider a word $e\in\mathbf{T}_{\tau}^{k}$ for large $k$ and consider the case of $\alpha_{M}(e\odot g)=\alpha_{m}(e)\alpha_{M}(g)$. We must, by the bounded length of $g$, have $\alpha_{M}(e)/\alpha_{m}(e)\sim 1$ surely for some uniform constant.
Analogously, when $\alpha_{M}(e\odot g)=\alpha_{M}(e)\alpha_{m}(g)$ we obtain $\alpha_{M}(e)/\alpha_{m}(e)\sim1$, uniformly, and hence $\overline{\psi}_{\tau}^{*t}(e)\overline{\psi}_{\sigma^{e}\tau}^{t}(g)\sim \overline{\psi}_{\tau}^t (e)\sim \overline{\psi}_{\tau}^{t}(e)\overline{\psi}_{\sigma^{e}\tau}^{t}(g)$.

Finally we note that the last two sums in (\ref{eq:theMonster}) correspond to rare events. 
More formally, the proportion of summands where $\alpha_M(e)/\alpha_m(e)\leq c<\infty$ for some uniform $c\in\R$, is almost surely insignificant. 
This follows directly from the assumption that the RIFS is not strictly self-affine, the independence of choices for every node and the uniform bounds for fixed $l$. Applying any of the various Central Limit Theorems and Large Deviation Principles for martingales and Galton-Watson processes, see e.g.\ Athreya and Vidyashankar~\cite{Athreya97} or Williams~\cite{WilliamsProbability}, we find that, as $k$ increases, the dominant terms are where $\alpha_{M}(e\odot g)=\alpha_{M}(e)\alpha_{M}(g)$ and $\alpha_{M}(e\odot g)=\alpha_{M}(e)\alpha_{m}(g)$ for $\alpha_M(e)/\alpha_m(e)\gg c$. In other words, the defect measured by $c_{k,l}(\tau)$ tends to $1$ almost surely. The direction of convergence can be deduced from the values of $C(e,g)$, as they depend on the value of $t$.
\end{proof}

We remark that the result in Lemma~\ref{martingaleFactorConvergesTo1} shows that the sequence $Y_{k}/\E(Y_{1})^{k}$ looks like a martingale `in the limit'. In the additive case, $c_{k,l}=1$ (surely) and one can show that $Y_{k}/\E(Y_{1})^{k}$ is a $\mathcal{L}^{2}$-bounded martingale, but to work in greater generality we will not employ this fact here and prove the lower bound by a branching argument.

\begin{lma}\label{lowerBoundProof}
For $s<s_{B}$, the sequence of random variables $\{Y_{k}^{s}\}$ diverges to $+\infty$ almost surely and hence the box-counting dimension of $F_{\tau}$ is bounded below by s. We conclude
\[
\dim_{B}F_{\tau}\geq s_{B},
\]
almost surely.
\end{lma}

\begin{proof}
Using the definition for the $\infty$-variable stopping set $\Xi_{\tau}^{\delta}$ and using an argument identical to Lemma~\ref{lma:mainLower} we can write for $s=s_{B}-\epsilon$ and some $C^{*}>0$ and some almost surely positive $C_{\tau}$, 
\begin{equation}\label{lowerBoundWriting}
M_{\delta}(F_{\tau})\geq C^{*}C_{\tau}\delta^{-(s_{B}-\epsilon/2)}\sum_{e\in\Xi_{\tau}^{\delta}}\overline{\psi}_{\tau}^{s_{B}-\epsilon/2}(e).
\end{equation}
But one can easily see that for any $\xi>0$, the random variable \[Z^{s}_{n}=\sum_{e\in\Xi_{\tau}^{\xi^{n}}}\overline{\psi}_{\tau}^{s_{B}-\epsilon/2}(e)\] is also an approximate martingale, c.f. (\ref{infinityQuasiMartingale}), for some analogous random variable $c_{k,l}$ with properties as in Lemma~\ref{martingaleFactorConvergesTo1},
\begin{equation}\label{anotherBoundBitesTheDust}
\E(Z^{s}_{k+l}\mid \mathcal{F}^{k})=c_{k,l}Z_{k}^{s}\E(Z_{l}^{s}).
\end{equation}
Now, for $s<s_{B}$ let $k$ be large enough such that 
\[
\E\left(\sum_{e\in\mathbf{T}_{\tau}^{k}}\overline{\psi}_{\tau}^{s}(e)\right)>1,
\]
choose $\xi>0$ such that $\xi<\overline\alpha^{k}$. Then
\begin{equation}\nonumber
\E(Z_{l}^{s})\geq \E\left(\sum_{e\in\mathbf{T}_{\tau}^{k}}\overline{\psi}_{\tau}^{s}(e)\right)>1,
\end{equation}
for all $l$. 
Now, using Lemma~\ref{martingaleFactorConvergesTo1}, we choose $l_{\sup}$ large enough such that $\E(c_{k,l_{\sup}})\E(Z_{l_{\sup}}^{s})>1$ for all $k\geq l_{\sup}$.
Using~(\ref{anotherBoundBitesTheDust}), we conclude that $\E(Z^{s}_{ql_{\sup}})$ increases exponentially as $q$ grows.
Similarly, since \[Z^{s}_{l_{\sup}\lfloor k/l_{\sup}\rfloor}\asymp Z^{s}_{k},\] there exists some $\beta_{1},\beta_{2}>1$ 
and a constant $D$ such that $D^{-1}\beta_{1}^{k}\leq\E (Z^{s}_{k})\leq D\beta_{2}^{k}$.

Consider the stopping trees $\Xi_{\tau}^{\xi^{n}}$. Since we are conditioning on non-extinction, a simple Borel-Cantelli argument 
shows that, almost surely, in every surviving branch there are infinitely many nodes where the branch splits into two or more
subbranches. For definiteness let $N=N(\tau)$ be the least integer such that the number of subbranches $\#\{\lambda_{i}\in\Xi_{\tau}^{\xi^{N}}\}>1$. We know 
that $\widehat{\tau}_{i}=\sigma^{\lambda_{i}}\tau$ are independent and identical in distribution and hence, for $\zeta_{1},\zeta_{2}>1$,
\begin{align}
&\mu\left\{\tau\in\mathcal{T} \mid \exists C>0 \text{ s.t.\ for all $n$, }C^{-1}\zeta_{1}^{n}\leq Z_{n}^{s}(\tau)\leq C\zeta_{2}^{n}\right\}\nonumber\\
&=1-\prod_{j=1}^{\#\{\lambda_{i}\in\Xi_{\tau}^{\xi^{N}}\}}\left(1-
\mu\left\{\tau\in\mathcal{T} \mid \exists C>0 \text{ s.t.\ for all $n$, }C^{-1}\zeta_{1}^{n}\leq Z_{n}^{s}(\widehat{\tau}_{j})\leq C\zeta_{2}^{n}\right\}
\right)\label{eqn:conditionalMeasure1}\\
&=1-\left(1-\mu\left\{\tau\in\mathcal{T} \mid \exists C>0 \text{ s.t.\ for all $n$, }C^{-1}\zeta_{1}^{n}\leq Z_{n}^{s}(\tau)\leq C\zeta_{2}^{n}\right\}\right)^{\#\{\lambda_{i}\in\Xi_{\tau}^{\xi^{N}}\}}.\label{eqn:conditionalMeasure2}\\
&\geq 1-(1-\mu\left\{\tau\in\mathcal{T} \mid \exists C>0 \text{ s.t.\ for all $n$, }C^{-1}\zeta_{1}^{n}\leq Z_{n}^{s}(\tau)\leq C\zeta_{2}^{n}\right\})^{2}\label{eqn:unconditional}.
\end{align}
Note that the measure $\mu$ in (\ref{eqn:conditionalMeasure2}) and (\ref{eqn:conditionalMeasure2}) is conditioned on $\#\lambda_{i}$. However, since  $\#\{\lambda_{i}\in\Xi_{\tau}^{\xi^{N}}\}\geq2$ for all such nodes, (\ref{eqn:unconditional}) holds unconditionally.
For $x\in[0,1]$, the only solutions to $x\geq1-(1-x)^{2}$ are $0$ and $1$, and we conclude that the probability that $Z_{n}^{s}$ (eventually) increases at least at exponential rate $\zeta_{1}>1$ and at most at rate $\zeta_{2}>1$ is $0$ or $1$. 
Letting $\zeta_{1}<\beta_{1}$ and $\zeta_{2}>\beta_{2}$, it is easy to see that the probability must be $1$ by noticing that 
$\underline{\alpha}Z_{n}^{s}\leq Z_{n+1}^{s}\leq \overline{\alpha}Z_{n}^{s}$ and 
\[
\mu\left\{\tau \mid Z_{n}^{s}(\tau)\geq (\zeta_{2}+\epsilon)^{n}\text{ or }Z_{n}^{s}(\tau)\leq (\zeta_{1}-\epsilon)^{n}\right\}\to 0\text{ as }n\to\infty,
\]
for sufficiently small $\epsilon>0$.
This, and the arbitrariness of $\zeta_{1}$ imply that there exists $\gamma>1$ and a random constant $D_{\tau}$ such that $Z_{n}^{s}\geq D_{\tau}\gamma^{n}$ almost surely.

We can now bound the expression in~(\ref{lowerBoundWriting}) for $\delta<\xi$, where $k$ is such that $\xi^{k+1}<\delta\leq \xi^{k}$, and redefining $C^{*}$ if necessary, by
\begin{align*}
M_{\delta}(F_{\tau})&\geq C^{*}C_{\tau}\delta^{-(s_{B}-\epsilon/2)}\sum_{e\in\Xi_{\tau}^{\xi^{k}}}\overline{\psi}_{\tau}^{s_{B}-\epsilon/2}(e)\\
&\asymp C^{*}C_{\tau}\delta^{-(s_{B}-\epsilon/2)} Z^{s_{B}-\epsilon/2}_{k}\\
&\geq C^{*}C_{\tau}D_{\tau}\delta^{-(s_{B}-\epsilon/2)}\gamma^{k}.
\end{align*}
And so
\begin{align*}
\liminf_{\delta\to0}\frac{\log M_{\delta}(F_{\tau})}{-\log\delta}
&\geq\liminf_{\delta\to0}\frac{\log\left(\delta^{-(s_{B}-\epsilon/2)}\gamma^{k}\right)}{-\log\delta}\\
&=s_{B}-\epsilon/2+\liminf_{\delta\to0}k\frac{\log\gamma}{-\log\delta}\\
&=s_{B}-\epsilon/2-\frac{\log\gamma}{\log\xi}\geq s_{B}-\epsilon/2.
\end{align*}
\end{proof}
Using the same idea one can extend Lemma~\ref{upperBoundProof} and drop the condition $s\leq s^{x}+s^{y}$ by picking $l_{\sup}$ large enough such that $c_{k,l_{\sup}}\E(Y_{l_{\sup}}^{s})<1$, we omit details.
Combining Lemmas \ref{upperBoundProof} and \ref{lowerBoundProof} we conclude that Theorem \ref{mainInfinityResult} holds.


\end{document}